\documentclass[a4paper,11pt]{article}
\usepackage[latin1]{inputenc}
\usepackage[english]{babel}
\usepackage{amsmath}
\usepackage{amsfonts}
\usepackage{amssymb}
\usepackage{epsfig}
\usepackage{amsopn}
\usepackage{amsthm}
\usepackage{color}
\usepackage{graphicx}
\usepackage{subfigure}
\usepackage{enumerate}
\newtheorem{theorem}{Theorem}[section]

\newtheorem*{theorem*}{Theorem}
\newtheorem*{lemma*}{Lemma}
\newtheorem*{remark*}{Remark}
\newtheorem*{definition*}{Definition}
\newtheorem*{proposition*}{Proposition}
\newtheorem*{corollary*}{Corollary}
\numberwithin{equation}{section}
\newcommand{\real}{\mathbb{R}}

\let\ced=\c         

\def\qed{\,\unskip\kern 6pt \penalty 500
\raise -2pt\hbox{\vrule \vbox to8pt{\hrule width 6pt
\vfill\hrule}\vrule}\par}
\definecolor{darkblue}{rgb}{0.05, .05, .65}
\definecolor{darkgreen}{rgb}{0.1, .65, .1}
\definecolor{darkred}{rgb}{0.8,0,0}
\newcommand{\beqn}{\begin{equation}}
\newcommand{\eeqn}{\end{equation}}
\newcommand{\bear}{\begin{eqnarray}}
\newcommand{\eear}{\end{eqnarray}}
\newcommand{\bean}{\begin{eqnarray*}}
\newcommand{\eean}{\end{eqnarray*}}
\begin{document}

\title{\huge \bf Radial equivalence and applications to the qualitative theory for a class of non-homogeneous reaction-diffusion equations}

\author{
\Large Razvan Gabriel Iagar\,\footnote{Departamento de Matem\'{a}tica
Aplicada, Ciencia e Ingenieria de los Materiales y Tecnologia
Electr\'onica, Universidad Rey Juan Carlos, M\'{o}stoles,
28933, Madrid, Spain, \textit{e-mail:} razvan.iagar@urjc.es}
\\[4pt] \Large Ariel S\'{a}nchez\,\footnote{Departamento de Matem\'{a}tica
Aplicada, Ciencia e Ingenieria de los Materiales y Tecnologia
Electr\'onica, Universidad Rey Juan Carlos, M\'{o}stoles,
28933, Madrid, Spain, \textit{e-mail:} ariel.sanchez@urjc.es}\\
[4pt] }
\date{}
\maketitle

\begin{abstract}
Some transformations acting on radially symmetric solutions to the following class of non-homogeneous reaction-diffusion equations
$$
|x|^{\sigma_1}\partial_tu=\Delta u^m+|x|^{\sigma_2}u^p, \qquad (x,t)\in\real^N\times(0,\infty),
$$
which has been proposed in a number of previous mathematical works as well as in several physical models, are introduced. We consider here $m\geq1$, $p\geq1$, $N\geq1$ and $\sigma_1$, $\sigma_2$ real exponents. We apply these transformations in connection to previous results on the one hand to deduce general qualitative properties of radially symmetric solutions and on the other hand to construct self-similar solutions which are expected to be patterns for the dynamics of the equations, strongly improving the existing theory. We also introduce mappings between solutions which work in the semilinear case $m=1$.
\end{abstract}

\


\smallskip


\section{Introduction}

The aim of this work is to construct a family of transformations acting on the class of radially symmetric solutions to non-homogeneous parabolic equations in the following general form
\begin{equation}\label{eq1.gen}
|x|^{\sigma_1}\partial_tu=\Delta u^m+|x|^{\sigma_2}u^p,
\end{equation}
posed for $(x,t)\in\real^N\times(0,\infty)$ and with exponents as follows
\begin{equation}\label{range.exp}
m\geq1, \qquad p\geq1, \qquad \sigma_1, \ \sigma_2\in\real.
\end{equation}
This equation involves a competition between the effects of the diffusion term (which conserves and spreads the total mass of a solution) and of the reaction term (which tends to increase the mass of the solution), leading to an important discussion with respect to the finite time blow-up. Moreover, the equation is strongly non-homogeneous, presenting two different weights depending on the space variable, a fact that makes its qualitative study more complicated but also much more interesting. In particular, a new competition, between the influence of ``inner" sets with $|x|$ small and the one of the ``outer" sets where $|x|$ is sufficiently large, adds up to the previous one to modify the geometry and the dynamics of the solutions.

\medskip

\noindent \textbf{Physical models}. Several particular cases of Eq. \eqref{eq1.gen} have been proposed in a number of physical models coming both from fluid dynamics as well as from the heat transfer or combustion. More precisely

$\bullet$ the semilinear case of Eq. \eqref{eq1.gen} with exponents $m=1$, $\sigma_2=0$ and $\sigma_1=q>0$ has been proposed as an approximate model for the flow in a channel of a fluid whose viscosity depends on the temperature $u$ of the fluid. The deduction of a more complex and realistic model stems from Ockendon's works \cite{Ock77, Ock79}, but as it is shown in \cite{La84} and then in \cite{Floater91}, Eq. \eqref{eq1.gen} with the above mentioned exponents is a reasonable approximation of the initial model that can be also handled both by analytical and numerical methods. In fact, Stuart and Floater developed in \cite{SF90} a numerical scheme to compute the formation of singularities (blow-up of the derivative and blow-up of the solutions) in this equation, giving numerical evidence that if $p>2$ blow-up might occur at interior points of the channel. Later, Floater \cite{Floater91} and Chan and Kong \cite{CK97} performed a deeper analytical study of both blow-up and quenching of solutions in this specific model.

$\bullet$ the general case of Eq. \eqref{eq1.gen} with $\sigma_1=\sigma_2\in[-2,0)$ is the equation satisfied by the temperature in a model of combustion in a medium with a heat source and whose thermal conductivity depends on the temperature, as proposed by Kurdyumov, Kurkina and their collaborators, see \cite{KKMS80, KK04, Ku05a, Ku05b} and references therein. In this model, self-similar solutions to Eq. \eqref{eq1.gen} represent the thermal structures that may exist in the nonlinear medium, whose understanding and classification is the aim of the study. Some of the self-similar solutions to this specific model are deduced in the above quoted works, but we will give a simpler and more general approach to study this particular case when $\sigma_1=\sigma_2$ by mapping it onto a well-studied, classical equation.

$\bullet$ the more general non-homogeneous porous medium equation with reaction in dimension $N\geq1$, that is,
\begin{equation}\label{non.PME}
\varrho(x)\partial_tu=\Delta u^m+A(x)u^p, 
\end{equation}
with general weights which are supposed to behave like pure powers as $|x|\to\infty$ is a well-known model for the propagation of thermal waves in a non-homogeneous medium, deduced by Kamin and Rosenau in \cite{KR81, KR82}. This model became of great mathematical interest more recently, see for example \cite{RV06, KRV10, IS14, IS16} and references therein for Eq. \eqref{non.PME} without a source. In particular, it has been noticed that, if $\varrho(x)=|x|^{\sigma_1}$, then $\sigma_1=-2$ is a critical exponent. A mathematical study of the case with source, but only with $m=1$ and $A(x)$ constant function, has been performed in \cite{dPRS13}.

\medskip

\noindent \textbf{Mathematical precedents}. The mathematical analysis of Eq. \eqref{eq1.gen} started, as expected, with the simplest case, that is, when $\sigma_1=\sigma_2=0$, where we are left with the classical reaction-diffusion equation
\begin{equation}\label{eq1.hom}
\partial_tu=\Delta u^m+u^p,
\end{equation}
which is now quite well-understood, see for example the monographs \cite{QS} for the case $m=1$ and \cite{S4} for the case $m>1$. In particular, the study focused on the phenomenon of finite time blow-up, in connection with both the range of the exponent $p>1$ and the decay at infinity of the initial condition $u_0$ triggering the evolution. Thus, a very important exponent is the \emph{Fujita-type exponent} $p_F=m+2/N$, limiting between the range $1<p<p_F$ of finite time blow up for any non-trivial initial condition $u_0\geq0$ and the range $p>p_F$ where global solutions may exist if the initial condition decays very fast as $|x|\to\infty$. These results have been generalized afterwards to the weighted reaction
\begin{equation}\label{eq1.wei}
\partial_tu=\Delta u^m+|x|^{\sigma}u^p,
\end{equation}
especially in the case $p>m$ where some techniques can be directly translated from the homogeneous case. After the works by Pinsky \cite{Pi97, Pi98} in the semilinear case $m=1$, it was Qi \cite{Qi98} and then Suzuki \cite{Su02} who established both the Fujita-type exponent and the second critical exponent (measuring the optimal decay rate of $u_0(x)$ as $|x|\to\infty$ when $p>p_F$ in order to have global solutions) for \eqref{eq1.wei} when $m>1$, while Andreucci and Tedeev \cite{AT05} established the blow-up rate with some limitations on the exponent $\sigma$. A series of works by Guo, Shimojo, Souplet et al. \cite{GLS, GS11, GLS13, GS18} addressed the question of the blow-up set of solutions to \eqref{eq1.wei} in the semilinear case $m=1$. In a different research direction, works such as \cite{FT00} and \cite{MS21} studied the way finite time blow-up occurs, again in the semilinear case. In these latter papers, some critical exponents are established (analogous to similar ones derived earlier for \eqref{eq1.hom}, see for example \cite[Chapter 4]{S4}) and it is shown that, when $p>1$ is not very large, solutions near the blow-up time converge to self-similar patterns, while for $p$ large there is no fixed blow-up rate and the phenomenon is more complex.

It has been then noticed for Eq. \eqref{eq1.wei} that its analysis might be performed with some analogies not only for $\sigma>0$ but going down to $\sigma\geq-2$. This case of singular potential is sometimes known as the \emph{Hardy equation} and its modern research had as starting point the classical work by Baras and Goldstein \cite{BG84} where the linear case of Eq. \eqref{eq1.wei} with $p=m=1$ and $\sigma=-2$ has been studied. The results therein limiting between existence and non-existence of solutions have been then extended to more general weights in \cite{CM99} and to the fast diffusion case $m<1$ in works such as \cite{GK03, GGK05, Ko04}. More recently, the Hardy equation has been strongly studied in the semilinear case $m=1$, see for example \cite{BSTW17, BS19, CIT21a, CIT21b, T20, HT21}. We stress here that the results in the papers \cite{Qi98, FT00, Su02, MS21} mentioned in the previous paragraph are also proved for $\sigma>-2$, but without including the limit case of equality.

Going back to our Eq. \eqref{eq1.gen}, Wang and Zheng \cite{WZ06} established the Fujita-type exponent under the following restrictions
\begin{equation}\label{Fujita}
p_F(\sigma_1,\sigma_2)=m+\frac{2+\sigma_2}{N+\sigma_1}, \qquad {\rm if} \ 0\leq\sigma_1\leq\sigma_2<p(\sigma_1+1)-1, \qquad p>m\geq1.
\end{equation}
Later on, the second critical exponent for \eqref{eq1.gen} has been derived in \cite{LD14, ZM14} (the latter extending the result up to the doubly nonlinear equation), namely
\begin{equation}\label{secondary}
\mu(\sigma_1,\sigma_2)=\frac{\sigma_2+2}{p-m}, \qquad {\rm provided} \ 0\leq\sigma_1\leq\sigma_2<\frac{p-m}{m-1}N, \qquad p>m>1.
\end{equation}
Let us recall here that by second critical exponent we understand the optimal exponent $\mu(\sigma_1,\sigma_2)>0$ such that, if $u_0(x)\sim C|x|^{-\mu}$ as $|x|\to\infty$ with $0<\mu<\mu(\sigma_1,\sigma_2)$, then any solution to \eqref{eq1.gen} with initial condition $u_0(x)$ blows up in finite time, while if $\mu>\mu(\sigma_1,\sigma_2)$, there exist global in time solutions. A different technical approach on Eq. \eqref{eq1.gen} and generalized also to the doubly nonlinear diffusion has been considered by Martynenko, Tedeev and their collaborators in a series of works \cite{MT07, MT08, MTS12} both for the case $\sigma_1=\sigma_2=-l$ with $l>0$ or with $\sigma_1=-l<0$ and $\sigma_2=0$. The authors of these works established smallness restrictions on some weighted integrals of the initial condition $u_0$ in order for global solutions to exist. Similar smallness conditions but obtained via comparison principles instead of weighted integral estimates were deduced by Meglioli and Punzo \cite{MP20, MP21} for a more general (and regular) unique weight $\varrho(x)$ replacing both $|x|^{\sigma_1}$ and $|x|^{\sigma_2}$ in Eq. \eqref{eq1.gen}.

\medskip

\noindent \textbf{Self-similar solutions}. A very important class of solutions to Eq. \eqref{eq1.gen} is formed by the radially symmetric self-similar solutions, which are solutions in one of the following three forms
\begin{subequations}\label{SSS}
\begin{equation}\label{forward}
u(x,t)=t^{\alpha}f(|x|t^{-\beta}), \qquad (x,t)\in\real^N\times(0,\infty),
\end{equation}
\begin{equation}\label{backward}
u(x,t)=(T-t)^{-\alpha}f(|x|(T-t)^{\beta}), \qquad (x,t)\in\real^N\times(0,T), \qquad T>0,
\end{equation}
\begin{equation}\label{exponential}
u(x,t)=e^{\alpha t}f(|x|e^{-\beta t}), \qquad (x,t)\in\real^N\times(-\infty,\infty),
\end{equation}
\end{subequations}
which are called respectively \emph{forward}, \emph{backward} and \emph{exponential} self-similar solutions. Notice that forward self-similar solutions are global in time, while backward self-similar solutions present a finite time blow-up at $t=T$. The exponential self-similar solutions are quite rare but interesting, appearing for specific, critical exponents and being solutions that can be in fact defined also backward in time, for any $t\in\real$; this is why they are also known in literature as \emph{eternal solutions}.

It has been noticed that self-similar solutions are frequently the prototype of a general solution to a nonlinear diffusion equation, in the sense that they encode many functional properties shared by general solution, and they also usually represent the pattern to which solutions tend as $t\to\infty$ or $t\to T$ in the blow-up case. Moreover, physical models also consider them as equilibrium states, thus performing an analysis of them it is a very important question when trying to understand the dynamics and the geometry of the solutions of a nonlinear diffusion equation. As an example related to Eq. \eqref{eq1.gen}, papers establishing the physical model from combustion such as \cite{KK04, Ku05a, Ku05b} also work on the self-similar solutions to the equation deduced from the model.

More recently, the authors developed a larger project of understanding the dynamics of Eq. \eqref{eq1.wei} starting from the classification of its self-similar solutions. Focusing at first on the range of exponents $1<p<m$, we have shown that both the occurence of the finite time blow-up, its sets and rates and the geometric form and evolution of the self-similar profiles depends strongly on the magnitude of $\sigma\geq-2$, see for example the results in \cite{ILS22, IMS21, IMS22, IS19, IS21a, IS23} and references therein. The very interesting case $p=m$ is critical and Eq. \eqref{eq1.wei} has been considered with $p=m$ in \cite{IS20, IS22}. As an outcome of all these developments, it has been proved at the level of self-similar solutions that their form is strongly influenced in the case of Eq. \eqref{eq1.wei} by the sign of the following constant:
\begin{equation}\label{const.L}
L:=\sigma(m-1)+2(p-1).
\end{equation}
Indeed, when $L>0$, backward self-similar solutions as in \eqref{backward} have been constructed, while if $L<0$, self-similar solutions are in forward form \eqref{forward}. Eternal, exponential self-similar solutions have been constructed in the special case $L=0$ in \cite{IS22a, IS22b, ILS22b}. We end up this presentation by introducing here a constant related to $L$ which will be very useful in the study of self-similar solutions to Eq. \eqref{eq1.gen}
\begin{equation}\label{const.L2}
L(\sigma_1,\sigma_2):=\sigma_2(m-1)+2(p-1)-\sigma_1(m-p).
\end{equation}

\medskip

\noindent \textbf{Brief description of our results}. In the present paper, we extend our analysis to the more general Eq. \eqref{eq1.gen} by means of a number of \textbf{transformations} mapping radially symmetric solutions to Eq. \eqref{eq1.gen} onto radially symmetric solutions to Eq. \eqref{eq1.wei} and in some cases, also onto the homogeneous equation \eqref{eq1.hom}. These mappings will be described in detail in \textbf{Section \ref{sec.transf}}, which is split into two parts: Subsection \ref{subsec.main} is devoted to the description of the transformation we use mostly in the sequel, while in Subsection \ref{subsec.other} three different transformations that we also consider of interest are introduced. We stress here that among the latter ones, mappings onto solutions to one-dimensional Fisher-KPP type equations are very useful in critical cases.

The rest of the paper is devoted to \textbf{a number of applications} of these transformations, having as starting point the already acquired knowledge about more particular equations such as \eqref{eq1.wei} or \eqref{eq1.hom}. More specifically, we identify in \textbf{Section \ref{sec.Fujita}} the Fujita-type exponent and then also the second critical exponent for Eq. \eqref{eq1.gen}, showing thus that, at least in the case of radially symmetric solutions, the technical restrictions on $\sigma_1$ and $\sigma_2$ given in \eqref{Fujita} and \eqref{secondary} can be removed. The forthcoming two chapters are dedicated to the classification of self-similar solutions (in any of the three forms \eqref{forward}, \eqref{backward} and \eqref{exponential}) to Eq. \eqref{eq1.gen}, depending on the relation between $p$ and $m$ but also on $\sigma_1$ and $\sigma_2$. Thus, \textbf{Section \ref{sec.equal}} classifies backward self-similar solutions for the limiting case $p=m$, presenting finite time blow-up, showing that they exist only if some inequality between $\sigma_1$ and $\sigma_2$ is satisfied. The longer \textbf{Section \ref{sec.lower}} is then devoted to the interesting range $1\leq p<m$, in which, depending on some relations between the exponents of Eq. \eqref{eq1.gen}, all three types of self-similar solutions may exist. In particular, it is shown that their dynamics depend strongly on the sign of the constant $L(\sigma_1,\sigma_2)$ given in \eqref{const.L2}, and their local behavior as $|x|\to0$ is also classified. In the same line of analysis, in the range $p>m$ a number of critical exponents are introduced in \textbf{Section \ref{sec.sigmaequal}} and their relevance for finite time blow-up of the solutions is exposed in the special case (of interest also in applications) when $\sigma_1=\sigma_2$. 

A specific \textbf{Section \ref{sec.heat}} is devoted to the semilinear case $m=1$, where there are some models of physical interest (as explained in the previous paragraphs) and where more precise results with respect to the range $m>1$ can be established. Both general well-posedness (at least in the framework of radially symmetric solutions) and large time behavior results for general solutions will be derived with the aid of our mappings from existing results on Hardy-H\'enon equations. Finally, \textbf{Section \ref{sec.crit}} deals with the critical exponent $\sigma_1=-2$, where the transformations we employ are totally different from the ones working in the range $\sigma_1>-2$. We are now in a position to state and explain our main results.

\section{The transformations}\label{sec.transf}

From now one, we deal with radially symmetric solutions $u(r,t)$, $r=|x|$, to Eq. \eqref{eq1.gen}. Such solutions satisfy the following simplified equation
\begin{equation}\label{eq1.rad}
r^{\sigma_1}\partial_tu=(u^m)_{rr}+\frac{N-1}{r}(u^m)_r+r^{\sigma_2}u^p,
\end{equation}
which is the radially symmetric version of Eq. \eqref{eq1.gen}. Let us make here the general convention that, since we only deal with radially symmetric solutions to Eq. \eqref{eq1.gen}, we can assume (for the continuity of the transformations) that the dimension $N$ is any real number, as it becomes just a parameter in Eq. \eqref{eq1.rad}.

\subsection{The main transformation}\label{subsec.main}

Proceeding as in \cite{IRS13}, we look for a change of variable of the general form
\begin{equation}\label{change}
u(r,t)=r^{\delta}w(z,\tau), \qquad z=r^{\theta}, \qquad \tau=Ct
\end{equation}
where $\delta$, $\theta$ are exponents to be determined and $C>0$. A straightforward calculation which is analogous to the one performed in \cite[Section 2.1]{IRS13} gives that $w(z,\tau)$ solves the following transformed equation
\begin{equation}\label{interm1}
\begin{split}
Cw_{\tau}&=\theta^2z^{[(m-1)\delta+2\theta-\sigma_1-2]/\theta}(w^m)_{zz}+\theta(2m\delta+\theta+N-2)z^{[(m-1)\delta+\theta-\sigma_1-2]/\theta}(w^m)_z\\
&+m\delta(m\delta+N-2)z^{[(m-1)\delta-\sigma_1-2]/\theta}w^m+z^{[\sigma_2-\sigma_1+\delta(p-1)]/\theta}w^p.
\end{split}
\end{equation}
In order to simplify \eqref{interm1}, we first impose the following condition
\begin{equation}\label{cond1}
(m-1)\delta+2\theta-\sigma_1-2=0,
\end{equation}
which reduces \eqref{interm1} to
\begin{equation}\label{eq2.gen}
\begin{split}
Cw_{\tau}&=\theta^2(w^m)_{zz}+\frac{\theta(2m\delta+\theta+N-2)}{z}(w^m)_z\\
&+\frac{m\delta(m\delta+N-2)}{z^2}w^m+z^{[\sigma_2-\sigma_1+\delta(p-1)]/\theta}w^p.
\end{split}
\end{equation}
Eq. \eqref{eq2.gen} is the starting point for the next simplifications leading to previously studied equations. The next step is to cancel out the third term in the right hand side. We let in \eqref{eq2.gen} (also fulfilling \eqref{cond1}), in a first step,
\begin{equation}\label{change1}
\delta=0, \qquad \theta=\frac{\sigma_1+2}{2}, \qquad \overline{N}=\frac{2(N+\sigma_1)}{\sigma_1+2}, \qquad \sigma=\frac{2(\sigma_2-\sigma_1)}{2+\sigma_1},
\end{equation}
and also, in a second step in order to remove constants from the coefficients,
\begin{equation}\label{change1.bis}
s=az, \qquad a=\theta^{-2/(\sigma+2)}, \qquad C=\theta^{2\sigma/(\sigma+2)},
\end{equation}
where $\sigma$ has been already defined in \eqref{change1}. Straightforward calculations show that, after applying \eqref{change1} and \eqref{change1.bis}, Eq. \eqref{eq2.gen} becomes
\begin{equation}\label{eq2.wei}
w_{\tau}=(w^m)_{ss}+\frac{\overline{N}-1}{s}(w^m)_s+s^{\sigma}w^p,
\end{equation}
which is nothing else that Eq. \eqref{eq1.wei} in radially symmetric variables, in dimension $\overline{N}$.

An important case of Eq. \eqref{eq1.gen}, also arising from physical models as explained in the Introduction, is the one with $\sigma_1=\sigma_2$ (see for example \cite{KKMS80, Ku05a, MT07, MTS12, MP20, MP21} and references therein), and we notice that our transformation \eqref{change1}-\eqref{change1.bis} maps its radially symmetric solutions onto radially symmetric solutions (in a different space dimension) to Eq. \eqref{eq1.hom}. More generally, employing this transformation for $\sigma_2\geq-2$ and $\sigma_1>-2$, which are the most usual limitations for the weights, we arrive to Eq. \eqref{eq1.wei} with
$$
\sigma=\frac{2(\sigma_2-\sigma_1)}{\sigma_1+2}=-2+\frac{2(\sigma_2+2)}{\sigma_1+2}\geq-2,
$$
falling into the range of $\sigma$ that has been studied thoroughly by the authors and their collaborators.

As another remark, the case of physical interest related to the models coming from fluid flow in channels \cite{Ock77, Ock79, SF90, CK97} has in our notation $\sigma_1>0$ (with a particular case of interest if $\sigma_1=1$), $\sigma_2=0$ and $m=1$. The resulting equation Eq. \eqref{eq1.gen} can be thus mapped to other already studied equations by employing our transformation \eqref{change1}-\eqref{change1.bis} to arrive to \eqref{eq2.wei} with
$$
\sigma=-\frac{2\sigma_1}{\sigma_1+2},
$$
which is similar to a transformation already considered in \cite{dPRS13}, which works in any dimension $N\geq1$.

Concerning the range of application of this change of variable, we observe that, on the one hand, it works very well when the initial dimension is $N\geq2$, since then also $\overline{N}\geq2$ (and in particular if $N=2$ then $\overline{N}=2$). On the other hand, it can also be used in dimension $N=1$ with the restriction $\sigma_1\geq0$, in order to ensure $\overline{N}\geq1$. In the forthcoming sections, we will mainly exploit this transformation in order to obtain completely new results or to extend and improve existing results on the radially symmetric solutions (and in particular, self-similar solutions) to Eq. \eqref{eq1.gen} by means of the already established knowledge on equations obtained through these mappings. But before examining these applications, let us consider some other available transformations that could be of use in some cases.

\subsection{Some more transformations}\label{subsec.other}

We gather in this chapter some more transformations. We stress here that the last two of them will be used in Section \ref{sec.crit} in order to obtain some interesting properties of solutions to Eq. \eqref{eq1.gen} with critical exponent $\sigma_1=-2$.

\medskip

\noindent \textbf{A second transformation}. We let in \eqref{change} (also fulfilling \eqref{cond1})
\begin{equation}\label{change2}
\delta=-\frac{N-2}{m}, \qquad \theta=\frac{m(N+\sigma_1)-N+2}{2m},
\end{equation}
and by applying once more \eqref{change1.bis}, Eq. \eqref{eq2.gen} is transformed again into \eqref{eq2.wei} but with the following new dimension and exponent $\sigma$
\begin{equation}\label{exp2}
\overline{N}=\frac{2[m(\sigma_1+2)-N+2]}{m(\sigma_1+N)-N+2}, \qquad \sigma=-\frac{2[m(\sigma_1-\sigma_2)+(N-2)(p-1)]}{m(\sigma_1+N)-N+2}.
\end{equation}
Let us notice that we can obtain once more the homogeneous case \eqref{eq1.hom}, that is, $\sigma=0$, if
$$
\sigma_2=\sigma_1+\frac{(N-2)(p-1)}{m},
$$
which generalizes the problem modelled and discussed in dimension $N=3$ in \cite{KKMS80}. More generally, if we let $\sigma_1>-2$ and $\sigma_2\geq-2$, with $N\geq2$, we get
$$
\sigma+2=\frac{2[N(m-p)+m\sigma_2+2p]}{m(\sigma_1+N)-N+2}\geq\frac{2(N-2)(m-p)}{m(\sigma_1+N)-N+2}\geq0
$$
provided $m\geq p$, hence we fall again on the range $\sigma\geq-2$. With respect to dimensions, we observe again that $N=2$ is mapped onto $\overline{N}=2$ and that $N=1$ is mapped onto $\overline{N}>2$. We also notice that, if $N=1$ and we assume $\sigma_2>-1$, which is a natural and rather standard condition (cf. \cite{ILS22}), we again find that $\sigma+2\geq0$, thus we are in the case that has been studied in previous literature. This specific transformation has been used in \cite{Floater91} to map Eq. \eqref{eq1.gen} in dimension $N=1$ and with $\sigma_2=0$, $\sigma_1>0$ into \eqref{eq2.wei} with
$$
\sigma=-\frac{2(\sigma_1+1-p)}{\sigma_1+2}.
$$
Finally, let us remark that $\overline{N}\geq1$ is equivalent to $N\leq[m(\sigma_1+4)+2]/(m+1)$ and if $\sigma_1=0$, we get as a particular case the self-map introduced in \cite{IS22}.

However, this transformation has a very serious drawback, in the sense that it involves a significant change of the properties of the initial condition of a solution from Eq. \eqref{eq1.rad} into Eq. \eqref{eq2.wei}, which might produce in the process solutions that fall out of the functional spaces or local behaviors for which the theory has been previously developed (in particular, they might for example become singular at $r=0$). This is why, its usefulness is more limited.

\medskip

\noindent \textbf{A third transformation. Euler type}. Setting in \eqref{interm1}
\begin{equation}\label{change3}
\delta=\frac{\sigma_1+2}{m-1}, \qquad C=\theta=1,
\end{equation}
we obtain the following partial differential equation whose right hand side is of Euler form (with $t=\tau$ and $z=r$ in this case)
\begin{equation}\label{interm2}
w_t=z^2(w^m)_{zz}+(2m\delta+N-1)z(w^m)_z+m\delta(m\delta+N-2)w^m+z^{\sigma_2-\sigma_1+\delta(p-1)}w^p.
\end{equation}
Let us notice at this point that this transformation applies well when $\sigma_2-\sigma_1+\delta(p-1)=0$, which is equivalent to $L(\sigma_1,\sigma_2)=0$, where $L(\sigma_1,\sigma_2)$ has been defined in \eqref{const.L2}. Introducing in \eqref{interm2} the standard change of variable for Euler equations $z=e^y$ and assuming that $L(\sigma_1,\sigma_2)=0$, we further obtain
\begin{equation}\label{eq2.eul}
\begin{split}
w_t&=(w^m)_{yy}+\frac{mN-N+2+2m(\sigma_1+1)}{m-1}(w^m)_y\\
&+\frac{m(\sigma_1+2)(mN-N+2+m\sigma_1)}{(m-1)^2}w^m+w^p,
\end{split}
\end{equation}
which is an equation with a double reaction in our range of parameters. This transformation generalizes the one introduced in \cite[Section 6]{IS22a}, and we notice that it is expected that the condition $L(\sigma_1,\sigma_2)=0$ be the critical connection between exponents allowing for \emph{eternal solutions} of the form \eqref{exponential}. In fact, if we go back to the transformation \eqref{change1} and we recall the value of $\sigma$ from \eqref{eq2.wei}, we notice that
$$
\sigma(m-1)+2(p-1)=\frac{2L(\sigma_1,\sigma_2)}{\sigma_1+\sigma_2},
$$
which suggests that the condition $L(\sigma_1,\sigma_2)=0$ for Eq. \eqref{eq1.gen} is mapped onto the condition $L=0$ for Eq. \eqref{eq1.wei}, where $L$ is the constant defined in \eqref{const.L}, and we already know from \cite{IS22a, ILS22b} that $L=0$ is the necessary and sufficient condition in Eq. \eqref{eq1.wei} to ensure the existence of eternal self-similar solutions in exponential form. Notice also that the \textbf{critical case $\sigma_1=-2$} (its criticality for the non-homogeneous porous medium equation follows for example from works such as \cite{KRV10, IS14, IS16}) works very well with this transformation by simply letting $\delta=0$ and $\theta=1$. The drawback of this transformation is that, in order to make use of it, one has to know the properties of solutions to equations of the form \eqref{eq2.eul}, which are not so much studied in literature.

\medskip

\noindent \textbf{The semilinear case $m=1$}. This is a special case where the transformations in \eqref{change1} and \eqref{change2} still work well if $\sigma_1>-2$, $\sigma_2\geq-2$. However, for the limiting case $\sigma_1=-2$ we can go deeper than in the previous paragraph by considering a similar change of variable. Indeed, since the condition $(m-1)\delta-\sigma_1-2=0$ becomes an identity and we are not forced to choose $\delta$ yet, we can fix in \eqref{interm1}
\begin{equation}\label{change4}
\delta=-\frac{\sigma_2+2}{p-1}, \qquad \theta=C=1, \qquad z=e^y,
\end{equation}
to get the following reversed Fisher-type equation with convection
\begin{equation}\label{interm3}
w_t=w_{yy}+\left[N-2-\frac{2(\sigma_2+2)}{p-1}\right]w_y-\frac{\sigma_2+2}{p-1}\left[N-2-\frac{\sigma_2+2}{p-1}\right]w+w^p.
\end{equation}
A priori \eqref{interm3} is not easy to handle, but we can still go further and remove the (linear) convection term by setting
\begin{equation}\label{change4.bis}
w(y,t)=\psi(y+Kt,t), \qquad K=N-2-\frac{2(\sigma_2+2)}{p-1}, \qquad \overline{y}=y+Kt
\end{equation}
to obtain the simplified equation
\begin{equation}\label{eq3.eul}
\psi_t=\psi_{\overline{y}\overline{y}}-\frac{\sigma_2+2}{p-1}\left[N-2-\frac{\sigma_2+2}{p-1}\right]\psi+\psi^p,
\end{equation}
which is a reversed Fisher-KPP type equation if $N>(\sigma_2+2p)/(p-1)$, or equivalently, $p>p_F(-2,\sigma_2)$, where we recall that $p_F(\sigma_1,\sigma_2)$ is defined in \eqref{Fujita}.

\section{Fujita-type and second critical exponents}\label{sec.Fujita}

In this section we consider $p>m\geq1$ and we establish both the Fujita-type exponent and the second critical exponent. The results are gathered in the next
\begin{theorem}\label{th.Fujita}
Let $\sigma_1>-2$ and $\sigma_2>-2$ if $N\geq2$, or $\sigma_1\geq0$ and $\sigma_2>-1$ if $N=1$. Then
\begin{enumerate}
  \item If $m<p\leq p_F(\sigma_1,\sigma_2)$, where $p_F(\sigma_1,\sigma_2)$ is defined in \eqref{Fujita}, then there is no non-trivial global in time radially symmetric solution to Eq. \eqref{eq1.gen}.
  \item Let now $p>p_F(\sigma_1,\sigma_2)$. If $u_0(|x|)$ is a radially symmetric function such that there exists $\delta>0$ with
\begin{equation}\label{interm4}
\liminf\limits_{|x|\to\infty}|x|^{\mu(\sigma_1,\sigma_2)}u_0(|x|)>\delta,
\end{equation}
where $\mu(\sigma_1,\sigma_2)$ is defined in \eqref{secondary}, then there is no global in time radially symmetric solution $u$ to Eq. \eqref{eq1.gen} such that $u(x,0)=u_0(x)$ for any $x\in\real^N$.
  \item Let now $p>p_F(\sigma_1,\sigma_2)$. Then for any $\mu>\mu^*(\sigma_1,\sigma_2)$ and $\delta>0$, there exists some $\epsilon>0$ (depending on $\delta$ and $\mu$) such that, for any radially symmetric function $u_0(|x|)$ such that
\begin{equation}\label{interm5}
u_0(|x|)\leq\min\{\epsilon,\delta|x|^{-\mu}\},
\end{equation}
there exists (at least) a global in time radially symmetric solution $u$ to Eq. \eqref{eq1.gen} such that $u(x,0)=u_0(x)$, for any $x\in\real^N$.
\item For any $p>p_F(\sigma_1,\sigma_2)$, there exists a global in time solution to Eq. \eqref{eq1.gen} in forward self-similar form \eqref{forward}.
\end{enumerate}
\end{theorem}
Let us remark at this point that the previous statement is equivalent to say that $p_F(\sigma_1,\sigma_2)$ is the Fujita-type exponent for Eq. \eqref{eq1.gen}, that the critical case $p=p_F(\sigma_1,\sigma_2)$ belongs to the blow-up range, and that $\mu(\sigma_1,\sigma_2)$ is the second critical exponent for Eq. \eqref{eq1.gen}. Thus, Theorem \ref{th.Fujita}, at least in the class of radially symmetric solutions, strongly generalizes previous results in papers such as \cite{LD14, WZ06, ZM14} since we completely get rid of the technical limitations on $\sigma_1$ and $\sigma_2$ stated in \eqref{Fujita} and \eqref{secondary}.
\begin{proof}
We start from the results in Qi \cite{Qi98} and Suzuki \cite{Su02} (see also references therein) concerning the Fujita-type exponent and the second critical exponent for Eq. \eqref{eq1.wei} when $p>m$. Let us notice that they cover all the range $\sigma>-2$ if $N\geq2$ (with the standard restriction $\sigma>-1$ if $N=1$).

\medskip

Let us first assume that $N\geq2$. In this case, we recall that the transformation \eqref{change1} maps solutions to Eq. \eqref{eq1.rad} onto radially symmetric solutions to Eq. \eqref{eq2.wei} with exponent $\sigma=2(\sigma_2-\sigma_1)/(2+\sigma_1)>-2$ and dimension $\overline{N}\geq2$ defined in \eqref{change1}. We then know that, if $m<p\leq m+(\sigma+2)/\overline{N}$, there is no non-trivial global in time solution to Eq. \eqref{eq2.wei}, according to \cite[Theorem 1.1]{Qi98}. Noticing that
$$
m+\frac{\sigma+2}{\overline{N}}=m+\frac{\sigma_2+2}{\sigma_1+N}=p_F(\sigma_1,\sigma_2),
$$
and the fact that \eqref{change1} acts only on the space variable, thus preserving finite time blow-up of solutions, we readily obtain that there are no global solutions to Eq. \eqref{eq1.rad} if $m<p<p_F(\sigma_1,\sigma_2)$. To prove the second and third statement of the theorem, we start from \cite[Theorem 1]{Su02}, which states that the second critical exponent to Eq. \eqref{eq2.wei} is $\mu^*=(\sigma+2)/(p-m)$, of course for $p>m+(\sigma+2)/\overline{N}$. Notice that, if $u_0(|x|)$ is a radially symmetric function such that $u_0(x)\sim\delta|x|^{-\mu}$ as $|x|\to\infty$, for some generic $\mu>0$, then transformation \eqref{change1} maps in into a function $w$ such that
$$
w(s)\sim\delta s^{-2\mu/(\sigma_1+2)}, \qquad {\rm as} \ s\to\infty.
$$
By imposing the condition that $2\mu/(\sigma_1+2)=\mu^*$, we obtain
$$
\mu=\frac{(\sigma_1+2)(\sigma+2)}{2(p-m)}=\frac{\sigma_2+2}{p-m}=\mu(\sigma_1,\sigma_2).
$$
It is then easy to derive parts 2 and 3 of Theorem \ref{th.Fujita} from \cite[Theorem 1, (b) and (c)]{Su02} taking into account that \eqref{change1} preserve either the finite time blow-up or the global existence in time of the solutions that are mapped and that \eqref{interm4}, respectively \eqref{interm5} are equivalent to the conditions in \cite[Theorem 1, (b)]{Su02}, respectively \cite[Theorem 1, (c)]{Su02}. Finally, if $N=1$, we notice that
$$
\overline{N}=\frac{2(\sigma_1+1)}{\sigma_1+2}\geq1,
$$
since $\sigma_1\geq0$, and a simple inspection of the proofs shows that the results in \cite{Qi98, Su02} apply for $\sigma>-\overline{N}$ also when taking $\overline{N}\geq1$ as a parameter in Eq. \eqref{eq2.wei} for radially symmetric solutions. Noticing that
\begin{equation}\label{interm6}
\sigma+\overline{N}=\frac{2(\sigma_2+1)}{\sigma_1+2}>0,
\end{equation}
the rest of the proof follows the same lines as for $N\geq2$.

\medskip

We are only left with the fourth statement, which follows by undoing the transformation \eqref{change1}-\eqref{change1.bis} to the radially symmetric self-similar solution to Eq. \eqref{eq1.wei} given in \cite[Theorem 1.2]{Qi98}. Indeed, it is shown in the latter reference that Eq. \eqref{eq1.wei} admits a self-similar solution in the form
$$
W(s,\tau)=\tau^{-\alpha}f(|s|\tau^{-\beta}), \qquad \alpha=\frac{\sigma+2}{\sigma(m-1)+2(p-1)}, \qquad \beta=\frac{p-m}{\sigma(m-1)+2(p-1)},
$$
provided $\sigma>-\overline{N}$, which also stays true when considering $\overline{N}$ as a real parameter in the equation of the self-similar profiles in \cite[Section 4]{Qi98}.
By taking into account \eqref{interm6} and undoing \eqref{change1} to the solution $W$, we find a solution in self-similar form as follows:
$$
U(x,t)=t^{-\alpha(\sigma_1,\sigma_2)}F(|x|t^{-\beta(\sigma_1,\sigma_2)}),
$$
where $C$ has been defined in \eqref{change1.bis} and
\begin{equation*}
\begin{split}
&F(\xi)=C^{-\alpha(\sigma_1,\sigma_2)}f\left(aC^{-\beta(\sigma_1,\sigma_2)}\xi^{(\sigma_1+2)/2}\right), \\
&\alpha(\sigma_1,\sigma_2)=\frac{2+\sigma_2}{L(\sigma_1,\sigma_2)}, \qquad \beta(\sigma_1,\sigma_2)=\frac{p-m}{L(\sigma_1,\sigma_2)},
\end{split}
\end{equation*}
and we recall that $a$ is defined in \eqref{change1.bis}, while the constant $L(\sigma_1,\sigma_2)$ is defined in \eqref{const.L2}.
\end{proof}

\section{Separate variable solutions in the case $p=m$}\label{sec.equal}

Another important feature of Eq. \eqref{eq1.gen} is the existence of self-similar, radially symmetric solutions in one of the forms \eqref{forward}, \eqref{backward} or \eqref{exponential} as discussed in the Introduction. In this section, we use our transformations to construct self-similar solutions in the critical case $p=m$, which in this particular case will be of separate variables, and find a relationship between $\sigma_1$, $\sigma_2$ and the parameters $m=p$ and $N$ limiting their existence. We start from the recent results obtained in \cite{IS22, IL22c}.
\begin{theorem}\label{th.pequalm}
Let $N\geq2$ and $-2<\sigma_1\leq\sigma_2$ or $N=1$ and $0\leq\sigma_1\leq\sigma_2$. Define
\begin{equation}\label{sigma.crit}
\sigma_{2,c}:=\sigma_1+\frac{[2(N-1)+\sigma_1](m-1)}{3m+1}.
\end{equation}
We have the following affirmations concerning existence and non-existence of self-similar solutions in backward form:
\begin{enumerate}
  \item For any $N\geq1$ and for any $\sigma_2$ such that $\sigma_1\leq\sigma_2<\sigma_{2,c}$, there exist separate variable solutions to Eq. \eqref{eq1.gen} presenting finite time blow-up, with the following form
\begin{equation}\label{sep.var}
u(x,t)=(T-t)^{-\alpha}F(|x|), \qquad \alpha=\frac{1}{m-1}.
\end{equation}
  \item For any $N\geq1$, there are no radially symmetric separate variable solutions provided that
\begin{equation}\label{nonexist.pm}
\sigma_2\geq\sigma_1+\frac{(m-1)(N+\sigma_1)}{m+1}>\sigma_{2,c}.
\end{equation}
  \item If $N>(m\sigma_1+4m+2)/(m+1)$, the non-existence result becomes sharp: there is no radially symmetric separate variable solution to Eq. \eqref{eq1.gen} if $\sigma_2\geq\sigma_{2,c}$.
\end{enumerate}
\end{theorem}
Notice that, if we let $\sigma_1=0$, we obtain exactly the critical exponent introduced in \cite{IS22} and the estimates for non-existence in \cite[Theorem 1.1]{IL22c}.
\begin{proof}
\noindent \textbf{Part 1.} Let first $\sigma_2\in[\sigma_1,\sigma_{2,c})$. By applying the transformation \eqref{change1}, we map radially symmetric solutions to Eq. \eqref{eq1.gen} onto radially symmetric solutions to Eq. \eqref{eq1.wei} with $\sigma=2(\sigma_2-\sigma_1)/(\sigma_1+2)\geq0$. We infer from \cite{IS22} that for any
\begin{equation}\label{interm7}
0\leq\sigma<\sigma_c:=\frac{2(m-1)(\overline{N}-1)}{3m+1},
\end{equation}
there exist separate variable solutions of the form
$$
w(s,\tau)=(T_0-\tau)^{-1/(m-1)}f(|x|),
$$
for suitable profiles $f$ satisfying a differential equation given in \cite{IS22}. We then observe that, in terms of the correspondences in \eqref{change1}, condition \eqref{interm7} is equivalent to
$$
0\leq\frac{2(\sigma_2-\sigma_1)}{2+\sigma_1}<\frac{2(m-1)}{3m+1}\left[\frac{2(N+\sigma_1)}{2+\sigma_1}-1\right],
$$
which is the same as
$$
\sigma_1\leq\sigma_2<\sigma_1+\frac{[2(N-1)+\sigma_1](m-1)}{3m+1}=\sigma_{2,c}.
$$
Finally, undoing the transformation \eqref{change1}, we get radially symmetric separate variable solutions with the form
$$
u(x,t)=w(ar^{\theta},Ct)=(T_0-Ct)^{-1/m-1}f(a|x|^{\theta})=(T-t)^{-1/(m-1)}F(|x|),
$$
which is the same as \eqref{sep.var}, where $a$, $C$ are defined in \eqref{change1.bis} and
$$
T_0=\frac{T}{C}, \qquad F(|x|)=C^{-1/(m-1)}f(a|x|^{\theta}), \qquad \theta=\frac{\sigma_1+2}{2}.
$$

\medskip

\noindent \textbf{Part 2.} This follows from the general non-existence result of \cite[Theorem 1.1]{IL22c}, proved with the aid of a Pohozaev identity, which in fact holds true for general separate variable solutions to Eq. \eqref{eq1.wei} with $p=m$. It states that, in any dimension $\overline{N}\geq1$, there are no separate variable solutions if
\begin{equation}\label{interm8}
\sigma\geq\frac{(m-1)\overline{N}}{m+1}.
\end{equation}
We then find, by replacing $\sigma$ and $\overline{N}$ from \eqref{change1}, that \eqref{interm8} becomes
$$
\frac{2(\sigma_2-\sigma_1)}{2+\sigma_1}\geq\frac{2(m-1)(N+\sigma_1)}{(m+1)(2+\sigma_1)},
$$
which leads to \eqref{nonexist.pm} after obvious simplifications.

\medskip

\noindent \textbf{Part 3.} A closer inspection of Step 4 of the proof of \cite[Theorem 1.1]{IL22c} reveals that, for the optimality in the non-existence range, and if assuming that $N$ is just a real parameter in a partial differential equation in radially symmetric variables, it is required to hold true that
$$
\sigma_c=\frac{2(m-1)(\overline{N}-1)}{3m+1}>\frac{2(m-1)}{m+1},
$$
which leads to $\overline{N}>(4m+2)/(m+1)$, a fact which also confirms the deductions made at a formal level in \cite{IS22}. Under this greatness condition on $\overline{N}$, non-existence of separate variable solutions (even if not radially symmetric) in Eq. \eqref{eq1.wei} holds true for any $\sigma\geq\sigma_c$. But the latter conditions on $\overline{N}$ and $\sigma$ together with the definition of $\overline{N}$ in \eqref{change1} give
$$
N>\frac{m\sigma_1+4m+2}{m+1}, \qquad \sigma_2\geq\sigma_1+\frac{[2(N-1)+\sigma_1](m-1)}{3m+1}=\sigma_{2,c},
$$
and this, together with the existence proved in Part 1, close the circle and show the sharpness of the non-existence range for $\sigma_2$ with respect to $\sigma_1$.
\end{proof}
Related to the case $p=m$, we can also obtain some results when one of the weights is the celebrated Hardy potential, that is, $K|x|^{-2}$, for some suitable constant $K>0$ which does not exceeds the optimal Hardy constant $K_*(N)=(N-2)^2/4$, in dimension $N\geq3$. This constant has been obtained in the classical work by Baras and Goldstein \cite{BG84} for $m=p=1$ as a limit between the range of existence and non-existence of solutions. We will thus deal with a slightly modified equation than Eq. \eqref{eq1.gen}, namely
\begin{equation}\label{eq1.Hardy}
|x|^{\sigma_1}u_t=\Delta u^m+\frac{K}{|x|^2}u^m,
\end{equation}
posed in dimension $N\geq 3$ and consider radially symmetric and compactly supported initial conditions $u(x,0)=u_0(x)$ such that $u_0\in C(\real^N)$, $u_0\geq0$. We then have the following
\begin{theorem}\label{th2.pequalm}
Let $N\geq3$, $\sigma_1>-2$ and let $u_0\in C(\real^N)$ be a radially symmetric, compactly supported function such that $u_0(x)\geq0$ for any $x\in\real^N$ and $u_0\not\equiv0$. Let then $K$ be such that $0<K<(N-2)^2/4$. Then there exists a unique self-similar solution to \eqref{eq1.Hardy} taking $u_0$ as initial condition as $t\to0$. Moreover, if in addition
\begin{equation}\label{lim1}
\lim\limits_{x\to0}|x|^{-(\sigma_1+2)/(m-1)}u_0(x)=+\infty,
\end{equation}
then we have a case of \emph{instantaneous blow-up} at $x=0$, in the sense that $\lim\limits_{x\to0}u(x,t)=+\infty$ for any $t>0$. On the contrary, if
\begin{equation}\label{lim2}
\limsup\limits_{x\to0}|x|^{-(\sigma_1+2)/(m-1)}u_0(x)<+\infty,
\end{equation}
then the solution $u(x,t)$ blows up only at $x=0$ in finite time $t=T\in(0,\infty)$ but not instantaneously. In both cases, the solution can be continued after the blow-up time for any $t>0$.
\end{theorem}
\begin{proof}
We notice that, if $\sigma_1>-2$ and $\sigma_2=-2$, by applying the transformation \eqref{change1}-\eqref{change1.bis} we are left with Eq. \eqref{eq2.wei} with $\sigma=-2$. We then apply previously established results on the equation
\begin{equation}\label{eq2.Hardy}
w_{\tau}=\Delta w^m+\overline{K}s^{-2}w^m, \qquad 0<\overline{K}<K_*(\overline{N})=\frac{(\overline{N}-2)^2}{4}.
\end{equation}
Notice then that \eqref{eq2.Hardy} is obtained via the transformation \eqref{change1} from Eq. \eqref{eq1.Hardy} with a constant $K=\overline{K}\theta^2$. In particular, the existence of radially symmetric solutions to the Cauchy problem follows directly from \cite[Proposition 1.2]{IS20}, provided
$$
K=\overline{K}\theta^2\leq\frac{(\overline{N}-2)^2}{4}\frac{(\sigma_1+2)^2}{4}=\frac{(N-2)^2}{4}=K_*(N),
$$
as claimed. The local behavior near $x=0$ follows from \cite[Theorems 1.3 and 1.4]{IS20} by noticing that the limiting power $s^{-2/(m-1)}$ in the variable $s=ar^{\theta}$ of the transformed equation is mapped into the limiting power
$$
r^{-2\theta/(m-1)}=r^{-(\sigma_1+2)/(m-1)},
$$
while the properties of the initial conditions remain unchanged by the transformation \eqref{change1}.
\end{proof}

\section{Self-similar solutions for $1\leq p<m$}\label{sec.lower}

The analysis of Eq. \eqref{eq1.wei} in this range of exponents was practically lacking from theory and some significant recent progress in its understanding has been achieved by the authors in a series of papers, see for example \cite{IS19, IS21a, ILS22, ILS22b, IMS21, IS23}, where different ranges related to the dimension $N$ and the sign of the constant $L$ in \eqref{const.L} were considered. The outcome of this analysis was quite unexpected, all the exponents having a strong influence in both the form of the self-similar solutions and of their profiles. We translate and generalize these results to Eq. \eqref{eq1.gen} by employing our transformations. The analysis will be split in this case with respect to the sign of the constant $L(\sigma_1,\sigma_2)$ defined in \eqref{const.L2}.

\medskip

\noindent \textbf{Case 1:} $\mathbf{L(\sigma_1,\sigma_2)<0}.$ We show in this case that forward self-similar solutions appear, which are global in time. More precisely, we have
\begin{theorem}\label{th1.psmall}
Let $m>1$, $p$, $\sigma_1$, $\sigma_2$ be such that
\begin{equation}\label{cond.exp1}
-2<\sigma_2<\sigma_1, \qquad 1\leq p<p_c(\sigma_1,\sigma_2):=m-\frac{(m-1)(\sigma_2+2)}{\sigma_1+2},
\end{equation}
in dimension $N\geq2$, adding up the restrictions $\sigma_1\geq0$, $\sigma_2>-1$ in dimension $N=1$. Then there exists a unique self-similar solution in forward form \eqref{forward}, where
$$
\alpha=-\frac{\sigma+2}{L(\sigma_1,\sigma_2)}, \qquad \beta=-\frac{m-p}{L(\sigma_1,\sigma_2)},
$$
such that its profile $f(\xi)$, with $\xi=|x|(\theta^2t)^{-\beta}$, $\theta=(\sigma_1+2)/2$, has the following local behavior at the origin
\begin{equation}\label{beh.Q1}
f(\xi)\sim\left[D-\frac{m-p}{m(N+\sigma_2)(2+\sigma_2)}\xi^{\sigma_2+2}\right]^{1/(m-p)}, \qquad \text{as} \ \xi\to0,
\end{equation}
where $D>0$ is a constant depending on $\sigma_1$ and $\sigma_2$, and is compactly supported at some $\xi_0\in(0,\infty)$ with
$$
f(\xi)>0 \ \text{for} \ \xi\in(0,\xi_0), \qquad f(\xi_0)=0, \qquad (f^m)'(\xi_0)=0.
$$
\end{theorem}
\begin{proof}
Our starting point is the statement of \cite[Theorem 1.1]{IMS21}, ensuring that there exists a unique radially symmetric self-similar solution in the form \eqref{forward} to Eq. \eqref{eq1.wei} having a compactly supported profile, with self-similar exponents
$$
\overline{\alpha}=-\frac{\sigma+2}{\sigma(m-1)+2(p-1)}, \qquad \overline{\beta}=-\frac{m-p}{\sigma(m-1)+2(p-1)}
$$
and local behavior as $\xi\to0$ given by
\begin{equation}\label{interm9}
\overline{f}(\overline{\xi})\sim\left[D(\sigma)-\frac{m-p}{m(\overline{N}+\sigma)(\sigma+2)}\overline{\xi}^{\sigma+2}\right]^{1/(m-p)},
\end{equation}
where we use the notation with bar for the variables related to Eq. \eqref{eq1.wei}, provided that $1\leq p<1-\sigma(m-1)/2$. We apply our transformation \eqref{change1}-\eqref{change1.bis} and we infer from the value of $\sigma$ in \eqref{change1} that the self-similar exponents are mapped to
$$
\alpha=-\frac{\sigma+2}{L(\sigma_1,\sigma_2)}, \qquad \beta=-\frac{m-p}{L(\sigma_1,\sigma_2)}
$$
and the limiting value of $p$ for the solution to exist changes into
$$
p_{\rm max}=1-\frac{\sigma(m-1)}{2}=\frac{m\sigma_1-(m-1)\sigma_2+2}{\sigma_1+2}=p_c(\sigma_1,\sigma_2).
$$
Thus, the forward self-similar solution with compact support exists provided $1\leq p<p_c(\sigma_1,\sigma_2)$, as stated. Notice that \eqref{cond.exp1} gives
$$
p_c(\sigma_1,\sigma_2)-1=\frac{(m-1)(\sigma_1-\sigma_2)}{\sigma_1+2}>0, \qquad m-p_c(\sigma_1,\sigma_2)=\frac{(m-1)(\sigma_2+2)}{\sigma_1+2}>0,
$$
whence $p_c(\sigma_1,\sigma_2)\in(1,m)$. With respect to the local behavior as $\xi\to0$, we apply \eqref{change1} and \eqref{change1.bis} to \eqref{interm9}. In particular, we get
\begin{equation*}
\begin{split}
\frac{m-p}{m(\overline{N}+\sigma)(\sigma+2)}&\overline{\xi}^{\sigma+2}=\frac{m-p}{m(\overline{N}+\sigma)(\sigma+2)}(s\tau^{-\overline{\beta}})^{\sigma+2}\\
&=\frac{(m-p)(\sigma_1+2)^2}{4m(N+\sigma_2)(\sigma_2+2)}\left(\theta^{-2/(\sigma+2)}r^{\theta}\theta^{-2\sigma\overline{\beta}/(\sigma+2)}t^{-\overline{\beta}}\right)^{\sigma+2}\\
&=\frac{m-p}{m(N+\sigma_2)(\sigma_2+2)}r^{\theta(\sigma+2)}t^{-\overline{\beta}(\sigma+2)}\theta^{-2\sigma\overline{\beta}}\\
&=\frac{m-p}{m(N+\sigma_2)(\sigma_2+2)}\theta^{-\beta(\sigma_2-\sigma_1)}(rt^{-\beta})^{\sigma_2+2}\\
&=\frac{m-p}{m(N+\sigma_2)(\sigma_2+2)}\xi^{\sigma_2+2},
\end{split}
\end{equation*}
where $\xi=|x|(\theta^{(\sigma_2-\sigma_1)/(\sigma_2+2)}t)^{-\beta}$. The proof is now complete.
\end{proof}
Let us remark here that this unique self-similar solution with compact support is global in time, but a close inspection of the classification of self-similar solutions in \cite{IMS21} plus an application of the transformation \eqref{change1} reveal that there are infinitely many self-similar solutions with the same exponents $\alpha$, $\beta$ and the same local behavior as $\xi\to0$ as in \eqref{interm9}, but with an increasing, unbounded behavior as $\xi\to\infty$. Furthermore, the unique compactly supported solution given by Theorem \ref{th1.psmall} suggests that compactly supported general solutions to Eq. \eqref{eq1.gen} should also exist globally in time (provided the comparison principle holds true, which is a difficult problem).

\medskip

\noindent \textbf{Case 2:} $\mathbf{L(\sigma_1,\sigma_2)>0}.$ This is the blow-up range, where self-similar solution in backward form \eqref{backward} are expected. We indeed have the following result:
\begin{theorem}\label{th1.plarge}
Let $m>1$, $1\leq p<m$ and assume that $\sigma_1>-2$, $\sigma_2>-2$ if $N\geq2$ or $\sigma_1\geq0$, $\sigma_2>-1$ if $N=1$ are such that $L(\sigma_1,\sigma_2)>0$.. Then there exist compactly supported, radially symmetric self-similar solutions to Eq. \eqref{eq1.gen} in the form \eqref{backward}, with positive self-similar exponents
$$
\alpha=\frac{\sigma+2}{L(\sigma_1,\sigma_2)}, \qquad \beta=\frac{m-p}{L(\sigma_1,\sigma_2)}.
$$
\end{theorem}
Together with this existence theorem, we can furthermore state a result classifying the local behavior in a neighborhood of the origin of the profiles $f(\xi)$, where $\xi=|x|(T-t)^{\beta}$. This depends on the magnitude of $\sigma_2$ with respect to $\sigma_1$ and will be made precise below. In what follows, by $C(m,N,p,\sigma_1,\sigma_2)$ we understand positive constants depending on the mentioned parameters.
\begin{theorem}\label{th2.plarge}
In the same conditions as in Theorem \ref{th1.plarge}, we have the following classification.
\begin{enumerate}
  \item There exists $K_0>0$ such that, if
\begin{equation}\label{interm10}
\sigma_2<\sigma_1+K_0\frac{\sigma_1+2}{2},
\end{equation}
then the self-similar profiles $f(\xi)$ present the following local behavior as $\xi\to0$:
\begin{equation}\label{beh1}
f(\xi)\sim\left\{\begin{array}{ll}\left[D+C(m,p,N,\sigma_1,\sigma_2)\xi^{\sigma_1+2}\right]^{1/(m-1)}, & {\rm if} \ \sigma_2\geq\sigma_1\\[1mm]
\left[D-\frac{m-p}{m(N+\sigma_2)(\sigma_2+2)}\xi^{\sigma_2+2}\right]^{1/(m-p)}, & {\rm if} \ -2<\sigma_2<\sigma_1\end{array}\right.
\end{equation}
  \item There exists at least a value $K^*\in(0,\infty)$ such that, if
\begin{equation}\label{interm11}
\sigma_2=\sigma_1+K^*\frac{\sigma_1+2}{2},
\end{equation}
then the self-similar profiles $f(\xi)$ present the following local behavior as $\xi\to0$:
\begin{equation}\label{beh2}
f(\xi)\sim C(m,N,p,\sigma_1,\sigma_2)\xi^{(\sigma_1+2)/(m-1)}.
\end{equation}
  \item There exists $K_1>0$ such that, if
\begin{equation}\label{interm12}
\sigma_2>\sigma_1+K_1\frac{\sigma_1+2}{2},
\end{equation}
then the self-similar profiles $f(\xi)$ present the following local behavior as $\xi\to0$:
\begin{equation}\label{beh3}
f(\xi)\sim C(m,N,p,\sigma_1,\sigma_2)\xi^{(\sigma_2+2)/(m-p)}.
\end{equation}
\end{enumerate}
\end{theorem}
Notice that the previous exponents in the local behaviors as $\xi\to0$ are coherent with the ones coming from the non-homogeneous porous medium equation without reaction, as shown in \cite{RV06}, where the typical exponent $\xi^{2}$ in the Barenblatt solutions is replaced by $\xi^{2+\sigma_1}$.
\begin{proof}[Proof of Theorems \ref{th1.plarge} and \ref{th2.plarge}]
The outcome of Theorem \ref{th1.plarge} follows immediately by applying our transformation \eqref{change1} to the self-similar solutions to Eq. \eqref{eq1.wei} given in \cite[Theorem 1.2]{ILS22}. With respect to the local behavior as $\xi\to0$, we translate to our solutions the classification established in \cite[Theorem 1.3]{ILS22}, which gives the local behavior in a neighborhood of the origin of radially symmetric self-similar solutions to Eq. \eqref{eq1.wei} with respect to the magnitude of $\sigma$. Indeed, the conditions related to the intervals of $\sigma$ in the statement of \cite[Theorem 1.3]{ILS22} (that is, $\sigma<\sigma_0$, $\sigma=\sigma^*$, respectively $\sigma>\sigma_1$) are translated into \eqref{interm10}, \eqref{interm11}, respectively \eqref{interm12}. With respect to the local behavior as $\xi\to0$, we take the three possible behaviors given in the Introduction of \cite{ILS22} and, retaking the convention of denoting by bar the variables related to Eq. \eqref{eq1.wei} and by $K$ a generic positive constant (which may change from one line to another), we notice that
$$
\overline{\xi}^2=K\left(ar^{\theta}t^{-\overline{\beta}}\right)^2=K\left(rt^{-\overline{\beta}/\theta}\right)^{2\theta}=K\left(rt^{-\beta}\right)^{\sigma_1+2},
$$
in the first case, leading to \eqref{beh1} in the case $\sigma_2\geq\sigma_1$, while
$$
\overline{\xi}^{2/(m-1)}=K\left(ar^{\theta}t^{-\overline{\beta}}\right)^{2/(m-1)}=K\left(rt^{-\beta}\right)^{(\sigma_1+2)/(m-1)},
$$
in the second case, which leads to \eqref{beh2}, and finally
$$
\overline{\xi}^{(\sigma+2)/(m-p)}=K\left(ar^{\theta}t^{-\overline{\beta}}\right)^{(\sigma+2)/(m-p)}=K\left(rt^{-\beta}\right)^{(\sigma+2)\theta/(m-p)}=K\left(rt^{-\beta}\right)^{(\sigma_2+2)/(m-p)},
$$
taking into account the expressions of $\theta$ and $\sigma$ given in \eqref{change1}. The latter thus leads to \eqref{beh3}.
\end{proof}

\noindent \textbf{Remark. An explicit solution for $p=1$}. The transformed equation Eq. \eqref{eq2.wei} allows for an interesting explicit solution for $p=1$ and in dimension $\overline{N}=1$, identified in \cite{IS19}, namely, in self-similar variable, 
\begin{equation}\label{explicit1}
\overline{f}(\overline{\xi})=\overline{\xi}^{2/(m-1)}\left[\frac{m-1}{2m(m+1)}-B\overline{\xi}^{\sigma}\right]^{1/(m-1)}_{+}, \qquad B=\frac{(m-1)^2}{m(\sigma+2)(m\sigma+m+1)},
\end{equation}
with $\sigma=\sqrt{2(m+1)}$. We can apply our second transformation \eqref{change2}-\eqref{exp2} to the solution whose profile is given in \eqref{explicit1} to obtain an explicit solution to Eq. \eqref{eq1.gen} when the exponents and dimension satisfy the following conditions
$$
p=1, \qquad N=2+\frac{m(\sigma_1+2)}{m+1}, \qquad \sigma_2=\sigma_1+\frac{m\sqrt{2(m+1)}(\sigma_1+2)}{m+1}.
$$
This explicit solution is self-similar and has the following form
\begin{equation}\label{explicit1.bis}
\begin{split}
u(x,t)&=(T-t)^{-1/(m-1)}a^{2/(m-1)}|x|^{(\sigma_1+2)/(m-1)}\\&\times\left[\frac{m-1}{2m(m+1)}-Ba^{\sqrt{2(m+1)}}(T-t)|x|^{\sigma_2-\sigma_1}\right]_{+}^{1/(m-1)},
\end{split}
\end{equation}
where $a$ is, as usual, defined in \eqref{change1.bis}. Notice that the solution defined in \eqref{explicit1.bis} presents the typical local behavior \eqref{beh2} as $|x|\to0$ and an interface.

\medskip

\noindent \textbf{Case 3:} $\mathbf{L(\sigma_1,\sigma_2)=0}.$ We are left with this critical case, in which it is easy to show (by simply inserting the ansatz \eqref{forward} or \eqref{backward} in Eq. \eqref{eq1.gen} and equate the powers of $t$, respectively $T-t$ in the three terms of Eq. \eqref{eq1.gen} to obtain an incompatible system) that Eq. \eqref{eq1.gen} does not admit any self-similar solutions in either forward or backward form. Nevertheless, we prove that there are exponential self-similar solutions.
\begin{theorem}\label{th1.pequal}
Assume that $\sigma_1>-2$ in dimension $N\geq2$ and $\sigma_1\geq0$ in dimension $N=1$ and that
$$
\sigma_2=\frac{\sigma_1(m-p)-2(p-1)}{m-1}.
$$
Then there exists a unique pair of exponents
\begin{equation}\label{interm13}
\alpha^*, \ \beta^*\in(0,\infty), \qquad \frac{\alpha^*}{\beta^*}=\frac{\sigma_1+2}{m-1},
\end{equation}
such that there exists a one-parameter family of compactly supported self-similar solutions in exponential form \eqref{exponential} to Eq. \eqref{eq1.gen}, whose profiles satisfy $f(0)>0$ and the following local behavior as $\xi\to0$:
\begin{equation}\label{beh.Q1.exp}
f(\xi)\sim\left[K-\frac{(m-1)^2}{m(\sigma_1+2)[N(m-1)+\sigma_1(m-p)-2(p-1)]}\xi^{(\sigma_1+2)(m-p)/(m-1)}\right]^{1/(m-p)}.
\end{equation}
Moreover, for any $\alpha\in(\alpha^*,\infty)$ and corresponding $\beta\in(0,\infty)$ connected to $\alpha$ by the ratio in \eqref{interm13}, there exists a one-parameter family of self-similar solutions in exponential form \eqref{exponential} to Eq. \eqref{eq1.gen} such that they behave as in \eqref{beh.Q1.exp} as $\xi\to0$ and have the following behavior as $\xi\to\infty$
\begin{equation}\label{beh.infty}
f(\xi)\sim C(m,p,\alpha,\sigma_1)\xi^{(\sigma_1+2)/(m-1)}(\log\,\xi)^{-1/(p-1)},
\end{equation}
where $C(m,p,\alpha,\sigma_1)>0$ is a positive constant (that can be made explicit).
\end{theorem}
\begin{proof}
The proof follows from merging our transformation \eqref{change1} with the results in \cite[Theorem 1.1]{ILS22b}. First of all, the connection between exponents $\alpha$ and $\beta$ given in \eqref{interm13} is a direct consequence of the ansatz \eqref{exponential}. Indeed, if we compute every term in Eq. \eqref{eq1.gen} by replacing $|x|=\xi e^{\beta t}$ and $u(x,t)=e^{\alpha t}f(\xi)$, and we only look at the factors involving the time variable, we readily get the equalities
$$
m\alpha-2\beta=\alpha+\sigma_1\beta=p\alpha+\sigma_2\beta,
$$
which gives $\alpha=(\sigma_1+2)\beta/(m-1)$, as claimed in \eqref{interm13}. Let us now look at \cite[Theorem 1.1]{ILS22b}. On the one hand, it applies to Eq. \eqref{eq2.wei} when $\sigma=-2(p-1)/(m-1)$, which gives
$$
\sigma_2=\frac{\sigma_1(m-p)-2(p-1)}{m-1}, \qquad {\rm or \ equivalently} \qquad L(\sigma_1,\sigma_2)=0,
$$
as desired. On the other hand, \cite[Theorem 1.1]{ILS22b} states the uniqueness of a pair of self-similar exponents and a one-parameter family of compactly supported solutions in exponential form to Eq. \eqref{eq2.wei} whose profiles have the local behavior
\begin{equation}\label{interm14}
\overline{f}(\overline{\xi})\sim\left[K-\frac{(m-1)^2}{2m[\overline{N}(m-1)-2(p-1)]}\overline{\xi}^{2(m-p)/(m-1)}\right]^{1/(m-p)}, \qquad {\rm as} \ \overline{\xi}\to0.
\end{equation}
We next show that, by taking into account that $\overline{\xi}=a\xi^{\theta}$ and inserting into \eqref{interm14} the expressions for $\overline{N}$ and $\theta$ given in \eqref{change1} and also the constant $a$ in \eqref{change1.bis}, easy calculations lead to the local behavior \eqref{beh.Q1.exp}. More precisely,
\begin{equation*}
\begin{split}
\overline{\xi}^{2(m-p)/(m-1)}&=a^{2(m-p)/(m-1)}\xi^{2\theta(m-p)/(m-1)}=\theta^{-2}\xi^{(\sigma_1+2)(m-p)/(m-1)}\\
&=\frac{4}{(\sigma_1+2)^2}\xi^{(\sigma_1+2)(m-p)/(m-1)},
\end{split}
\end{equation*}
while
$$
\frac{(m-1)^2}{2m[\overline{N}(m-1)-2(p-1)]}=\frac{(m-1)^2(\sigma_1+2)}{4m[N(m-1)+\sigma_1(m-p)-2(p-1)]}.
$$
By multiplying the previous two equalities and inserting their outcome into \eqref{interm14}, we arrive to \eqref{beh.Q1.exp}. In the same way, starting from the local behavior at infinity of the unbounded self-similar profiles given in \cite[Theorem 1.1]{ILS22b}, that is,
$$
\overline{f}(\overline{\xi})\sim \overline{C}(m,p,\overline{\alpha})\overline{\xi}^{2/(m-1)}(\log\,\overline{\xi})^{-1/(p-1)}, \qquad {\rm as} \ \overline{\xi}\to\infty
$$
and applying the same transformations as in the previous considerations, we get the behavior \eqref{beh.infty}.
\end{proof}

\noindent \textbf{Remark}. Observe that the exponential self-similar solutions live for any $t\in(-\infty,\infty)$, thus they are also known in literature as \emph{eternal solutions}. The rescaling
\begin{equation}\label{resc}
u_{\lambda}(x,t)=\lambda u\left(\lambda^{-(m-1)/(\sigma_1+2)}x,t\right)
\end{equation}
maps solutions to Eq. \eqref{eq1.gen} with $L(\sigma_1,\sigma_2)=0$ into other solutions of the same equation. At the level of self-similar solutions, if moreover we let $\lambda=e^{\alpha t_0}$ for some $t_0\in\real$, \eqref{resc} writes as
$$
u_{\lambda}(x,t)=e^{\alpha(t+t_0)}f(|x|e^{-\beta(t+t_0)}),
$$
thus the rescaled versions of an exponential self-similar solution are just translations in time (either forward or backward, depending on whether $t_0>0$ or $t_0<0$). The existence of one-parameter families of exponential self-similar solutions in Theorem \ref{th1.pequal} has to be then understood as a uniqueness result modulo the rescaling \eqref{resc} which in fact gives the same solution delayed with different times.

\medskip

We are left with the limiting case $\sigma_2=-2$ (with $\sigma_1>-2$), which gives $\sigma=-2$ by applying the transformation \eqref{change1}. We then have
\begin{theorem}
Let $m>1$, $1<p<m$, $N\geq3$, $\sigma_1>-2$ and $\sigma_2=-2$. Then there exists a unique self-similar solution of the form
$$
u(x,t)=f(\xi), \qquad \xi=|x|t^{-1/(\sigma_1+2)},
$$
such that its profile $f(\xi)$ is compactly supported at some $\xi_0\in(0,\infty)$ in the sense that $f(\xi_0)=0$, $f(\xi)>0$ for any $\xi\in(0,\xi_0)$ and $(f^m)'(\xi_0)=0$, and has a logarithmic singularity at $\xi=0$
\begin{equation}\label{beh.Q1.lim}
f(\xi)\sim\left[-\frac{(m-p)(\sigma_1+2)^2}{4m(N-2)}\ln\,\xi+K\right]^{1/(m-p)}, \qquad {\rm as} \ \xi\to0.
\end{equation}
\end{theorem}
\begin{proof}
We readily notice that, by applying the transformation \eqref{change1}-\eqref{change1.bis} to Eq. \eqref{eq1.gen} with $\sigma_2=-2$, we arrive to Eq. \eqref{eq2.wei} with $\sigma=-2$, independent of the value of $\sigma_1>-2$. We next infer from \cite[Theorem 1.1]{IS23} that Eq. \eqref{eq2.wei}, posed in dimension $\overline{N}>2$ (even if taken as a real parameter, as it follows by an inspection of its proof) admits a unique compactly supported self-similar solution in the form
$$
w(s,\tau)=\overline{f}(\overline{\xi}), \qquad \overline{\xi}=|s|\tau^{-1/2}
$$
and such that
\begin{equation}\label{interm15}
\overline{f}(\overline{\xi})\sim\left[-\frac{m-p}{m(\overline{N}-2)}\ln\,\overline{\xi}+K\right]^{1/(m-p)}, \qquad {\rm as} \ \overline{\xi}\to0.
\end{equation}
Observe first that, if $N\geq3$, then
$$
\overline{N}=\frac{2(N+\sigma_1)}{\sigma_1+2}\geq\frac{2(\sigma_1+3)}{\sigma_1+2}=2+\frac{2}{\sigma_1+2}>2,
$$
hence the previous statement applies to any dimension $N\geq3$. We then notice that, modulo constants that do not play any role when taking logarithms, we have $\overline{\xi}\sim\xi^{\theta}$, hence, by applying \eqref{change1} to the local behavior in \eqref{interm15} and noticing that
$$
\frac{(m-p)\theta}{m(\overline{N}-2)}=\frac{(m-p)(\sigma_1+2)^2}{4m(N-2)},
$$
we obtain the local behavior \eqref{beh.Q1.lim}. Moreover, with respect to the self-similar exponent, we have
$$
|s|\tau^{-1/2}=ar^{\theta}\tau^{-1/2}=a(r\tau^{-1/2\theta})^{\theta}=aC^{-1/2}(rt^{-1/(\sigma_1+2)})^{\theta},
$$
hence the new self-similarity exponent in Eq. \eqref{eq1.gen} becomes $1/(\sigma_1+2)$, as claimed.
\end{proof}
Notice also that the singularity at $\xi=0$ is integrable, hence this solution can be considered as a weak solution in $L^1$ (and any other $L^p$ space with $1<p<\infty$), despite being unbounded. We refer the reader to our work \cite{IS23} for similar considerations and omit the details here.

\section{The special case $\sigma_1=\sigma_2$. Improved results}\label{sec.sigmaequal}

Throughout this section, we will consider $\sigma_1=\sigma_2\in(-2,\infty)$. As explained in the Introduction, this specific case has been considered in a number of both physical and mathematical works, and a number of results on it are now available. Moreover, all the results in previous sections also hold true. However, the fact that Eq. \eqref{eq1.gen} is mapped through the change of variable \eqref{change1} into the well studied homogeneous equation Eq. \eqref{eq1.hom} allows us to extract more information from it. To this end, let us introduce a number of critical exponents. The first one, known just as \emph{critical exponent}, is given by
\begin{equation}\label{pcsigma}
p_c(\sigma_2):=\frac{m(N+\sigma_2)}{N-2}, \qquad {\rm provided} \ N>2,
\end{equation}
while the second one is known as the \emph{Sobolev exponent}
\begin{equation}\label{pssigma}
p_s(\sigma_2):=\frac{m(N+2\sigma_2+2)}{N-2}, \qquad {\rm provided} \ N>2.
\end{equation}
The third critical exponent has been identified for the first time by Joseph and Lundgren in relation with some quasilinear elliptic problems in \cite{JL} and will be called the \emph{Joseph-Lundgren exponent}, which in our case will have the form
\begin{equation}\label{pjlsigma}
p_{JL}(\sigma_2):=\frac{m[N^2-8N+4-2\sigma_2^2-2(N+2)\sigma_2+2(\sigma_2+2)\sqrt{(\sigma_2+2)(2N+\sigma_2-2)}]}{(N-2)(N-10-4\sigma_2)},
\end{equation}
which is defined in dimension $N>10+4\sigma_2$. Finally, a higher critical exponent will be considered, namely, the \emph{Lepin exponent}, which has been identified for the first time by Lepin \cite{Le90} for the case $m=1$. In our notation, will have the rather tedious expression
\begin{equation}\label{plsigma}
p_L(\sigma_2):=1+\frac{3m(\sigma_2+2)+\sqrt{L}}{2(N-10-4\sigma_2)}, \qquad {\rm provided} \ N>10+4\sigma_2,
\end{equation}
where
$$
L=4(m-1)^2(N-10-4\sigma_2)^2+2(m-1)(5m-4)(\sigma_2+2)(N-10-4\sigma_2)+9m^2(\sigma_2+2)^2.
$$
In all these four expressions above, we consider by convention the critical exponents equal to $+\infty$ in dimensions that are smaller than the ones considered in the definitions. We will give below a sequence of results showing how these critical exponents influence on the qualitative behavior of radially symmetric solutions to Eq. \eqref{eq1.gen}. The first theorem puts into evidence the influence of the critical and Sobolev exponents.
\begin{theorem}\label{th.plargem1}
Let $m>1$, $N\geq2$ and $\sigma_1=\sigma_2\in(-2,\infty)$ (or $N=1$ and $\sigma_1=\sigma_2\in(0,\infty)$) and let $p$ be such that $m<p<p_s(\sigma_2)$. Then the following properties hold true.
\begin{enumerate}
\item There exists a decreasing radially symmetric self-similar solution in backward form \eqref{backward} to Eq. \eqref{eq1.gen} (modulo the choice of $T>0$), with self-similar exponents
$$
\alpha=\frac{1}{p-1}, \qquad \beta=\frac{p-m}{(\sigma_2+2)(p-1)}.
$$
\item For $p>p_c(\sigma_2)$ (and not limited to $p_s(\sigma_2)$), there exists an explicit singular stationary solution to Eq. \eqref{eq1.gen}
\begin{equation}\label{stat.sol}
u(x,t)=K|x|^{-(\sigma_2+2)/(p-m)}, \qquad K=\left[\frac{m(N-2)(\sigma_2+2)(p-p_c(\sigma_2))}{(p-m)^2}\right]^{1/(p-m)}.
\end{equation}
\item Let $u$ be a radially symmetric solution to Eq. \eqref{eq1.gen} which blows up at a finite time $T\in(0,\infty)$. Then blow-up is \emph{complete} if $m<p\leq p_s(\sigma_2)$, that is, $u(t)\equiv\infty$ for any $t>T$. The same holds true even for $p>p_s(\sigma_2)$, provided that the blow-up of $u$ at time $t=T$ takes place in more points than the origin $x=0$.
\end{enumerate}
\end{theorem}
\begin{proof}
\noindent \textbf{Part 1}. This follows easily from \cite[Theorem 4, Section 1, Chapter IV]{S4}. Indeed, the quoted reference gives the existence of a monotone decreasing, radially symmetric self-similar solution to Eq. \eqref{eq1.hom} with self-similar exponents
$$
\overline{\alpha}=\frac{1}{p-1}, \qquad \overline{\beta}=\frac{p-m}{2(p-1)}.
$$
Undoing the transformation \eqref{change1} entails $\alpha=\overline{\alpha}$ and
$$
\beta=\frac{\overline{\beta}}{\theta}=\frac{p-m}{(\sigma_2+2)(p-1)}.
$$
as desired

\medskip

\noindent \textbf{Part 2}. This follows by direct calculation and we omit the details. We can also obtain the same formula \eqref{stat.sol} by employing the transformation \eqref{change1}-\eqref{change1.bis} on the stationary solution (given for example in \cite[Section 5]{GV97})
$$
w(s)=\overline{K}s^{-2/(p-m)}, \qquad \overline{K}=\left[\frac{2m}{p-m}\left(\overline{N}-2-\frac{2m}{p-m}\right)\right]^{1/(p-m)},
$$
to Eq. \eqref{eq2.wei} with $\sigma=0$ and with dimension parameter $\overline{N}$ defined in \eqref{change1}.

\medskip

\noindent \textbf{Part 3}. Let $u$ be a radially symmetric solution to Eq. \eqref{eq1.gen} which blows up at some time $T\in(0,\infty)$. We transform it by applying \eqref{change1}-\eqref{change1.bis} into a solution $w(s,\tau)$ to Eq. \eqref{eq2.wei} in dimension $\overline{N}\geq1$ having as blow-up time $T_1:=(\sigma_2+2)^2T/4\in(0,\infty)$. According to \cite[Theorem 5.1]{GV97}, if $m<p\leq p_s:=m(\overline{N}+2)/(\overline{N}-2)$ blow-up of $w$ is complete, and the same happens if $p>p_s$ provided that the blow-up set of $w$ at $t=T$ is not the singleton $\{0\}$. We reach our conclusion by undoing the transformation and noticing that
$$
\frac{m(\overline{N}+2)}{\overline{N}-2}=\frac{m(N+2\sigma_2+2)}{N-2}=p_s(\sigma_2),
$$
as claimed.
\end{proof}
The two higher critical exponents $p_{JL}(\sigma_2)$ and $p_L(\sigma_2)$ play an important role in the existence and form of the self-similar solutions and the possibility of continuation after the blow-up time. More precisely, we have:
\begin{theorem}\label{th3.plarge}
Let $m>1$, $N>2$ and $\sigma_1=\sigma_2\in(-2,\infty)$ and let $p$ be such that $p_s(\sigma_2)<p$. Then the following properties hold true.
\begin{enumerate}
  \item If $p_s(\sigma_2)<p<p_{JL}(\sigma_2)$, then there exists an infinite sequence of radially symmetric self-similar solutions in backward form \eqref{backward} to Eq. \eqref{eq1.gen}
$$
u_n(x,t)=(T-t)^{-1/(p-1)}f_n(|x|(T-t)^{-\beta}), \qquad \beta=\frac{p-m}{(\sigma_2+2)(p-1)},
$$
such that the profiles $f_n$ satisfy
$$
f_n(\xi)\sim C\xi^{-(\sigma_2+2)/(p-m)}, \qquad {\rm as} \ \xi\to\infty.
$$
If $N> 10+4\sigma_2$ and $p_{JL}(\sigma_2)\leq p<p_L(\sigma_2)$, there exists at least one self-similar solution with the same form and properties. All these self-similar solutions can be continued for times $t>T$, in the sense that $u(\cdot,t)\in L^{\infty}(\real^N)$ for $t>T$.
  \item If $u_0(|x|)=u(|x|,0)$ for any self-similar solution $u$ as in the previous part, and if we consider the rescaling
$$
u_{0,\lambda}(|x|)=\lambda u_0\left(\lambda^{-(m-1)/(\sigma_2+2)}|x|\right), \qquad \lambda>0,
$$
then there exists at least a radially symmetric solution $u_{\lambda}$ to Eq. \eqref{eq1.gen}. Moreover, for $\lambda>1$, $u_{\lambda}(\cdot,t)\in L^{\infty}(\real^N)$ for any $t>0$ (that is, the solution is global), while for $\lambda>1$, $u_{\lambda}$ presents a complete blow-up at time $T_{\lambda}\in(0,\infty)$.
\end{enumerate}
\end{theorem}
\begin{proof}
The first part follows directly by applying the transformations \eqref{change1}-\eqref{change1.bis} to the self-similar solutions to Eq. \eqref{eq1.hom} given in \cite[Theorem 12.1]{GV97} for $p_s<p<p_{JL}$, respectively \cite[Theorem 12.2]{GV97} for $p_{JL}\leq p<p_{L}$, where the corresponding exponents
$$
p_s=\frac{m(\overline{N}+2)}{\overline{N}-2}, \qquad p_{JL}=m\left[1+\frac{4}{\overline{N}-4-2\sqrt{\overline{N}-1}}\right], \qquad \overline{N}\geq11,
$$
and
$$
p_L=1+\frac{3m+\sqrt{(m-1)^2(\overline{N}-10)^2+2(m-1)(5m-4)(\overline{N}-10)+9m^2}}{\overline{N}-10},
$$
also for $\overline{N}\geq11$, are respectively the Sobolev, Joseph-Lundgren and Lepin exponents for Eq. \eqref{eq1.gen} in dimension $\overline{N}\geq1$. The second part follows in a similar way from the outcome of \cite[Theorem 14.1]{GV97}. We omit here the details, as they are very similar to the ones in previous proofs.
\end{proof}

\section{The semilinear case $m=1$}\label{sec.heat}

This section is devoted to applications of our transformations to the heat equation with (possibly) two weights, that is, letting $m=1$ but any possible $\sigma_1$, $\sigma_2$ in Eq. \eqref{eq1.gen}. As explained in the Introduction, this case is strongly related to models from applied sciences, in particular from the fluid flow in channels according to \cite{Ock77, Ock79, SF90, Floater91, CK97}. Moreover, from the mathematical point of view, some results in the case $\sigma_2=0$ but $\sigma_1<0$ have been obtained in \cite{dPRS13}, where transformations which are particular cases of the ones we consider here have been introduced and used. We shall give here more general results which extend some of the ones given already in the above mentioned works. It is a well-known fact that the properties of the solutions to Eq. \eqref{eq2.wei} with $m=1$ differ quite strongly with respect to the sign of $\sigma$, which also leads to differences in the properties of Eq. \eqref{eq1.gen} between the cases $\sigma_1>\sigma_2$ and $\sigma_1<\sigma_2$. The first theorem is related to self-similar solutions and self-similar blow-up behavior. To state it, for any generic $\sigma>-2$ let us introduce the Sobolev and Joseph-Lundgren exponents
\begin{subequations}\label{critexp.heat}
\begin{equation}\label{Sobolevexp.heat}
p_s(\sigma)=\frac{N+2+2\sigma}{N-2}, \qquad {\rm for} \ N>2
\end{equation}
\begin{equation}\label{Josephexp.heat}
p_{JL}(\sigma)=1+\frac{2\sigma+4}{N-4-\sigma-\sqrt{(2N+\sigma-2)(\sigma+2)}}, \qquad {\rm for} \ N>10+4\sigma,
\end{equation}
\end{subequations}
with the convention that the two critical exponents are equal to $+\infty$ in lower space dimensions. We then have:
\begin{theorem}\label{th.selfheat}
Let $m=1$, $\sigma_1>-2$, $\sigma_2\geq-2$ in dimension $N\geq2$, with the further restriction $\sigma_1>0$ if $N=1$.

\medskip

\noindent (a) Assume $1<p<p_s(\sigma_2)$. Then, if $\sigma_2\geq\sigma_1$, there are no self-similar solutions to Eq. \eqref{eq1.gen}, while if $\sigma_2<\sigma_1$, there exist radially symmetric self-similar solutions of the form
\begin{equation}\label{SSheat}
u(x,t)=(T-t)^{-\alpha}f(|x|(T-t)^{-1/(\sigma_1+2)}), \qquad \alpha=\frac{\sigma_2+2}{(p-1)(\sigma_1+2)},
\end{equation}
with a decreasing profile $f$ such that
\begin{equation}\label{SSprof}
f(\xi)\sim K\xi^{-(\sigma_2+2)/(p-1)}, \qquad {\rm as} \ \xi\to\infty.
\end{equation}

\noindent (b) Assume $N\geq3$ and $p_s(\sigma_2)<p<p_{JL}(\sigma_2)$. Then there exist infinitely many radially symmetric self-similar solutions with the same form as in part (a).

\noindent (c) Assume now $N>10+4\sigma_2$ and $p>p_{JL}(\sigma_2)$. Then, for any sufficiently large natural number $n$, there exists a radially symmetric, positive and radially decreasing solution $u_n$ to Eq. \eqref{eq1.gen} blowing up at a finite time $T\in(0,\infty)$ and only at the origin $x=0$ with a prescribed blow-up rate
\begin{equation}\label{burate}
\lim\limits_{t\to T}(T-t)^{2n/L(N,p,\sigma_1)}u_n(0,t)=K>0,
\end{equation}
where
\begin{equation}\label{interm16}
L(N,p,\sigma_1)=\frac{(N-2)(p-p_s(\sigma_1))-\sqrt{M(N,p,\sigma_1)}}{2+\sigma_1}
\end{equation}
and
\begin{equation}\label{interm17}
M(N,p,\sigma_1)=(p-1)^2N^2-4(p-1)(p\sigma_1+3p-1)N+4p\sigma_1(\sigma_1+2p+2)+20p^2-8p+4.
\end{equation}
\end{theorem}
Notice that part (c) of Theorem \ref{th.selfheat} is a characterization of a blow-up of Type II, that is, where solutions can blow up at the same time $T$ with different blow-up rates for the same exponents in the equation.
\begin{proof}
(a) We infer from \cite[Theorem A, Part (a)]{FT00} that Eq. \eqref{eq2.wei} admits at least one radially symmetric self-similar solution with decreasing profile if $1<p<p_s(\sigma)$ and $\sigma\in(-2,0)$, and no proper radially symmetric self-similar solution at all if $\sigma\geq0$ (except, in the case $\sigma=0$, for the constant solution). We reach the conclusion by undoing the change of variable \eqref{change1} and noticing that
$$
p_s(\sigma)=\frac{\overline{N}+2\sigma+2}{\overline{N}-2}=\frac{N+2\sigma_2+2}{N-2},
$$
and that $\sigma\geq0$ is equivalent with $\sigma_2\geq\sigma_1$. Moreover, the blow-up self-similar solutions to Eq. \eqref{eq2.wei} have the form
$$
w(s,\tau)=(T_0-\tau)^{-(\sigma+2)/2(p-1)}f(|x|(T_0-\tau)^{-1/2}),
$$
with
$$
f(\xi)\sim K\xi^{-(\sigma+2)/(p-1)}, \qquad {\rm as} \ \xi\to\infty.
$$
Undoing the change of variables \eqref{change1}-\eqref{change1.bis}, we readily find the desired form of exponents and decay of the profile for $u$ self-similar solution to Eq. \eqref{eq1.gen} as stated in \eqref{SSheat} and \eqref{SSprof} (with a new blow-up time which is $T=T_0/C$ and $C$ is defined in \eqref{change1.bis}).

\medskip

\noindent (b) This is just an application of our transformation \eqref{change1}-\eqref{change1.bis} to the outcome of \cite[Theorem A, Part (b)]{FT00}.

\medskip

\noindent (c) The finite time blow-up of Type II for Eq. \eqref{eq2.wei} in the semilinear case $m=1$ but with $\sigma>-2$ and $p>p_{JL}(\sigma)$ has been thoroughly described in the recent work \cite{MS21}. More precisely, \cite[Theorem 1.1]{MS21} states that, for any sufficiently large natural number $n$, there exists at least one radially symmetric and radially decreasing positive solution $w_n$ to Eq. \eqref{eq2.wei} blowing up at some time $T_0\in(0,\infty)$ and at $s=0$, such that their blow-up rate is given by
$$
\lim\limits_{\tau\to T_0}(T_0-\tau)^{2n/\overline{L}(\overline{N},p)}w_n(0,\tau)=\overline{K},
$$
where
$$
\overline{L}(\overline{N},p):=(\overline{N}-2)(p-1)-4-\sqrt{\overline{M}(\overline{N},p)},
$$
and
$$
\overline{M}(\overline{N},p):=(p-1)^2(\overline{N}-10)(\overline{N}-2)-8(p-1)(\overline{N}-4)+16.
$$
Let us notice first that, taking into account the definition of $p_{JL}(\sigma)$ in \eqref{Josephexp.heat} and the expressions of $\sigma$ and $\overline{N}$ in \eqref{change1}, $p_{JL}(\sigma)$ is mapped into $p_{JL}(\sigma_2)$ by undoing \eqref{change1}. Moreover, the same argument gives that
$$
\overline{M}(\overline{N},p)=\frac{M(N,p,\sigma_1)}{(\sigma_1+2)^2}, \qquad \overline{L}(\overline{N},p)=L(N,p,\sigma_1),
$$
according to the expressions given in \eqref{interm16} and \eqref{interm17}. Thus, we find the blow-up rate in \eqref{burate}, with the mention that the new blow-up time for $u(x,t)$ is $T=T_0/C$, where $C>0$ is the constant in front of the time variable made precise in \eqref{change1.bis}.
\end{proof}
As a remark, the very interesting paper \cite{MS21} gives more precise asymptotic estimates on the solutions to Eq. \eqref{eq2.wei} with $m=1$ and their blow-up patterns both in the inner layer formed in a small neighborhood of the unique blow-up point $s=0$ and in bounded regions, leading to a very deep description of the asymptotics of the solutions $w_n$ as $\tau\to T_0$. All them can be mapped into a similar description of the behavior near the blow-up time and point for the solutions $u_n$ to Eq. \eqref{eq1.gen}. We refrain from entering these very technical calculations here.

Another remark is that the simplified model appearing in fluid flows through channels \cite{Floater91, SF90, CK97} leads to Eq. \eqref{eq1.gen} with $N=1$, $m=1$, $\sigma_1=q>0$ and $\sigma_2=0$, thus the mapping \eqref{change1}-\eqref{change1.bis} applies and leads to Eq. \eqref{eq2.wei} with $\sigma\in(-2,0)$, an equation known as the \emph{Hardy equation} and which became recently fashionable, according to the big number of recently published papers on it (see for example \cite{BSTW17, T20, CIT21a, CIT21b} and references therein). In particular, Theorem \ref{th.selfheat} applies to this case, but some more specific results can be obtained. We may thus exploit our main transformation and the already well developed theory of the Hardy equation to get the next result in a more general case including the above mentioned physically interesting equation.
\begin{theorem}\label{th.wp}
Let $m=1$ and assume that either $\sigma_1>-2$ in dimension $N\geq2$ or $\sigma_1>0$ in dimension $N=1$, and that $-2<\sigma_2<\sigma_1$ if $N\geq2$ or $-1<\sigma_2<\sigma_1$ if $N=1$. Then the following statements about radially symmetric solutions to Eq. \eqref{eq1.gen} hold true.
\begin{enumerate}
  \item Let $u_0\in C_0(\real^N)$ be a radially symmetric initial condition, where
$$
C_0(\real^N)=\left\{f:\real^N\mapsto\real: \ f \ {\rm continuous}, \ \lim\limits_{|x|\to\infty}f(x)=0\right\}.
$$
Then there exists a unique radially symmetric solution $u(x,t)$ to Eq. \eqref{eq1.gen}, defined on a maximal interval $t\in[0,T)$, such that $u(x,0)=u_0(x)$ for any $x\in\real^N$ and such that $u(t)\in C_0(\real^N)$ for any $t\in(0,T)$.
  \item The same well-posedness as in the first part also holds true in the weighted space
$$
L^q(\real^N;|x|^{\sigma_1}):=\left\{f:\real^N\mapsto\real: \int_{\real^N}|x|^{\sigma_1}f(x)\,dx<\infty\right\},
$$
provided
\begin{equation}\label{lq}
q>\max\left\{\frac{p(N+\sigma_1)}{N+\sigma_2},\frac{(p-1)(N+\sigma_1)}{2+\sigma_2}\right\}.
\end{equation}
  \item Let $p>1+(2+\sigma_2)/(N+\sigma_1)$. Then the unique solution with initial data $u_0\in C_0(\real^N)$, respectively $u_0\in L^q(\real^N)$ as in the two previous items, is global (that is, $T=\infty$) provided that either $u_0(x)\leq C(1+|x|)^{-(2+\sigma_2)/(p-1)}$ or $\|u_0\|_{q}$ is small enough.
\end{enumerate}
\end{theorem}
\begin{proof}
\noindent \textbf{Part 1.} Let $u_0\in C_0(\real^N)$ and define the function $w_0(s):=u_0(ar^{\theta})$, $r=|x|$, with $\theta$ defined in \eqref{change1} and $a$ defined in \eqref{change1.bis}. We notice that $w_0\in C_0(\real)$. Introduce also $\sigma<0$ and $\overline{N}$ as defined in \eqref{change1} and observe that
$$
\sigma+\overline{N}=\frac{2(\sigma_2-\sigma_1)}{2+\sigma_1}+\frac{2(N+\sigma_1)}{2+\sigma_1}=\frac{2(N+\sigma_2)}{2+\sigma_1}>0,
$$
hence we can apply \cite[Theorem 1.1,(i)]{BSTW17} and deduce that there exists a solution $w(s,\tau)$ to Eq. \eqref{eq2.wei} which is continuous and vanishes at infinity, defined on a maximal interval $\tau\in[0,T_0]$ for some $T_0\in(0,\infty]$. We undo the transformation \eqref{change1}-\eqref{change1.bis} and define
$$
u(r,t)=w(s,\tau), \qquad r=\left(\frac{s}{a}\right)^{1/\theta}, \qquad t=\frac{\tau}{C},
$$
where $a$, $C$, $\theta$ are defined as in \eqref{change1} and \eqref{change1.bis}. Then $u$ is a radially symmetric solution to Eq. \eqref{eq1.gen} with $u(r,0)=u_0(r)$ by construction, and defined on the maximal interval $(0,T)$ with $T=T_0/C$. Moreover, $u(t)\in C_0(\real^N)$, since the previous change of variable does not affect either the property of continuity or the vanishing at infinity since $\theta>0$ and thus $s\to\infty$ is equivalent to $r\to\infty$.

\medskip

\noindent \textbf{Part 2.} The same construction as in the proof of Part 1 gives a natural candidate to the solution, by applying \cite[Theorem 1.1, (ii)]{BSTW17} to Eq. \eqref{eq2.wei}, which ensures well-posedness of the latter equation in $L^q(\real^{\overline{N}})$ provided
$$
q>\max\left\{\frac{p\overline{N}}{\sigma+\overline{N}},\frac{(p-1)\overline{N}}{2+\sigma}\right\},
$$
which leads to the lower bound \eqref{lq} if we take into account the expressions of $\sigma$ and $\overline{N}$ given in \eqref{change1}. We only have to check that, by undoing the transformation, we are left in the weighted space $L^q(\real^N;|x|^{\sigma_1})$. But this follows from a simple change of variable in an integral, more precisely
\begin{equation*}
\begin{split}
\infty&>\int_{\real}w^q(s,\tau)s^{\overline{N}-1}\,ds=\int_{\real}w^q(s,\tau)(ar^{\theta})^{\overline{N}-1}a\theta r^{\theta-1}\,dr\\
&=K(\theta)\int_{\real}u^q(r,t)r^{\theta\overline{N}-1}\,dr=K(\theta)\int_{\real}u^q(r,t)r^{\sigma_1+N-1}\,dr\\
&=K(\theta)\int_{\real^N}|x|^{\sigma_1}u^q(|x|,t)\,dx,
\end{split}
\end{equation*}
which shows that $u(t)\in L^q(\real^N;|x|^{\sigma_1})$ for any $t\in(0,T)$ with $T=T_0/C$ as in Part 1. This completes the proof.

\medskip

\noindent \textbf{Part 3.} This follows now readily from \cite[Theorem 1.3]{BSTW17} and the previous arguments in Part 1 and Part 2. We omit the details, as they are completely similar to the global existence result proved in Theorem \ref{th.Fujita} for $m>1$.
\end{proof}
We end this section with a result of large time behavior of global solutions.
\begin{theorem}\label{th.asympt}
Let $m=1$ and assume that either $\sigma_1>-2$ in dimension $N\geq2$ or $\sigma_1>0$ in dimension $N=1$, and that $-2<\sigma_2$ if $N\geq2$ or $-1<\sigma_2$ if $N=1$. Assume that $p>p_F(\sigma_1,\sigma_2)$, where $p_F(\sigma_1,\sigma_2)$ is defined in \eqref{Fujita}. Let $\omega\in L^{\infty}(\real^N)$ be a homogeneous function of degree zero and with $\|\omega\|_{\infty}$ sufficiently small and define $\varphi(x)=\omega(x)|x|^{-(\sigma_2+2)/(p-1)}$
\begin{enumerate}
  \item There exists a self-similar solution $U(x,t)$ in forward form \eqref{forward} to Eq. \eqref{eq1.gen}, with self-similarity exponents
\begin{equation}\label{interm18}
\alpha=\frac{\sigma_2+2}{(p-1)(\sigma_1+2)}, \qquad \beta=\frac{1}{\sigma_1+2},
\end{equation}
with initial condition $\varphi(x)$ taken in the sense of distributions as $t\to0$.
  \item Assume now that $\sigma_2<\sigma_1$ and let $u_0\in C_0(\real^N)$ be an initial condition such that for any $x\in\real^N$ we have
$$
u_0(x)\leq C(1+|x|^2)^{-(\sigma_2+2)/2(p-1)}, \qquad u_0(x)=\omega(x)|x|^{-(\sigma_2+2)/(p-1)}, \qquad {\rm for} \ |x|\geq R,
$$
for some $R$ sufficiently large. Then we have the following large time behavior
\begin{equation}\label{asympt.beh}
\|u(t)-U(t)\|_{\infty}\leq Ct^{-(\sigma_2+2)/2(p-1)-\delta}, \qquad {\rm for \ any} \ t>0,
\end{equation}
for some $\delta>0$.
  \item In the same conditions as in Part 2, there exist positive constants $K_1$ and $K_2$ such that
$$
K_1t^{-(\sigma_2+2)/(p-1)}\leq\|u(t)\|_{\infty}\leq K_2t^{-(\sigma_2+2)/(p-1)}, \qquad {\rm for \ any} \ t>0.
$$
  \item In the same conditions as in Part 2, we consider initial conditions $u_0\in C_0(\real^N)$ such that
$$
u_0(x)\leq C(1+|x|)^{-\gamma}, \qquad u_0(x)=\omega(x)|x|^{-\gamma}, \qquad {\rm for} \ |x|\geq R, \ \frac{\sigma_2+2}{p-1}<\gamma<N+\sigma_1.
$$
Then there exist positive constants $K_1$ and $K_2$ such that
$$
K_1t^{-\gamma/2}\leq \|u(t)\|_{\infty}\leq K_2t^{-\gamma/2}, \qquad {\rm for \ any} \ t>0.
$$
\end{enumerate}
\end{theorem}
\begin{proof}[Sketch of the proof]
The proof of the first item follows directly by applying our main transformation \eqref{change1}-\eqref{change1.bis} to the forward self-similar solutions to Eq. \eqref{eq2.wei} with $m=1$ given in \cite[Theorem 1.4]{BSTW17} for $\sigma<0$ (which covers the case $\sigma_1>\sigma_2$) and generalized later to any $\sigma>-2$ in \cite[Theorem 1.11]{CIT21b}. The fact that the self-similar exponents are given by the expression \eqref{interm18} follows from the general formulas
$$
\alpha=\frac{\sigma_2+2}{L(\sigma_1,\sigma_2)}, \qquad \beta=\frac{p-1}{L(\sigma_1,\sigma_2)},
$$
and the fact that for $m=1$ we have $L(\sigma_1,\sigma_2)=(\sigma_1+2)(p-1)$, see \eqref{const.L2}. The above mentioned Theorems ensure then that there exist forward self-similar solutions to Eq. \eqref{eq2.wei} with initial condition $\overline{\varphi}(s)=\overline{\omega}(s)s^{-(\sigma+2)/(p-1)}$, and we notice by undoing the transformation \eqref{change1}-\eqref{change1.bis} that
$$
s^{-(\sigma+2)/(p-1)}=(ar^{\theta})^{-2(\sigma_2+2)/[(p-1)(\sigma_1+2)]}=\theta^{-2/(p-1)}r^{-(\sigma_2+2)/(p-1)},
$$
which gives the desired correspondence of decays as $|x|\to\infty$. The second, third and fourth items follow readily from the outcome of \cite[Theorem 1.5]{BSTW17}, which requires that $\sigma<0$ in Eq. \eqref{eq2.wei}, that is, $\sigma_2<\sigma_1$, and we omit here the details. Up to our knowledge, a similar large time behavior result for Eq. \eqref{eq2.wei} with $\sigma>0$ is still lacking, although we conjecture that it should be true.
\end{proof}
Notice that both Theorems \ref{th.wp} and \ref{th.asympt} apply in particular to the case $\sigma_1>0$, $\sigma_2=0$ which is related to physical models, as explained in the Introduction, giving in particular some well-posedness and large time behavior results for it.

\section{The critical exponent $\sigma_1=-2$}\label{sec.crit}

We end this series of applications by considering the critical exponent $\sigma_1=-2$, for which it is obvious that our main transformation \eqref{change1}-\eqref{change} does not work. As discussed in Subsection \ref{subsec.other}, there are transformation of Euler type that can be used in this case. Assume first that $m>1$. We then deduce from \eqref{change3} that we have to take $\delta=0$ and $C=\theta=1$ in \eqref{eq2.gen}. We then obtain an equation of reaction-convection-diffusion type with a weight on the reaction term. In general, the theory of such equations is still missing from literature, but there is a specific case, namely when $\sigma_1=\sigma_2=-2$, when the weight is removed and we are left with the equation
\begin{equation}\label{interm19}
w_t=(w^m)_{yy}+(N-2)(w^m)_y+w^p, \qquad y=\ln\,r=\ln\,|x|,
\end{equation}
where $w(y,t)=u(|x|,t)$. Let us recall here that this specific critical case $\sigma_1=\sigma_2=-2$ has been also considered in \cite{KKMS80} as arising from a model of combustion in a medium whose thermal conductivity is temperature-dependent. We devote the next result to its analytical study.
\begin{theorem}\label{th.Eul.large}
Let $m>1$ and $\sigma_1=\sigma_2=-2$. Then
\begin{enumerate}
  \item If $m\leq p<m+1$, then any non-trivial radially symmetric solution to Eq. \eqref{eq1.gen} blows up at a finite time $T\in(0,\infty)$.
  \item Let now $p>m+2$. Then there exists $\delta>0$ sufficiently small such that, if $u_0(|x|)$ is a radially symmetric initial condition that satisfies
\begin{equation}\label{interm20}
\|u_0\|_{(p-m)/2;-N}:=\left(\int_{\real^N}|x|^{-N}u_0(x)^{(p-m)/2}\,dx\right)^{2/(p-m)}<\delta,
\end{equation}
there exists a radially symmetric solution $u$ to Eq. \eqref{eq1.gen} such that $u(x,0)=u_0(x)$, which is global in time and moreover it satisfies
$$
\|u(t)\|_{\infty}\leq Kt^{-1/(p-1)}, \qquad t>0,
$$
for some constant $K>0$.
  \item In the same notation and conditions as in Part 2, if furthermore
\begin{equation}\label{interm21}
\|u_0\|_{1;-N}:=\int_{\real^N}|x|^{-N}u_0(x)\,dx<\infty,
\end{equation}
then there exists $K_1>0$ such that $\sup\{\|u(t)\|_{1;-N}:t>0\}<K_1$.
\end{enumerate}
\end{theorem}
\begin{proof}
\noindent \textbf{Part 1.} As discussed at the beginning of the section, if we let $w(y,t)=u(|x|,t)$, $y=\ln\,|x|$, Eq. \eqref{eq1.rad} is mapped into the reaction-convection-diffusion equation \eqref{interm19}. Since we are in the case when the exponent of the convection term is equal to $m$, we are in the framework of the study performed by Suzuki in \cite{Su98}, with $a=N-2$ and in space dimension one. We then infer from \cite[Theorem 1]{Su98} that, for $m\leq p<\min\{m+2,m+1\}=m+1$, all nontrivial solutions to Eq. \eqref{interm19} blow up in finite time. Since the transformation obviously does not affect the time behavior, we infer that the same holds true for radially symmetric solutions to Eq. \eqref{eq1.gen}, completing the proof of the first item.

\medskip

\noindent \textbf{Part 2.} Let $u_0$ be a radially symmetric function satisfying \eqref{interm20} and define $w_0(y)=u_0(r)$, $y=\ln\,r$, where, as usual, $r=|x|$. We then use an obvious change of variable to get that
\begin{equation*}
\begin{split}
\|w_0\|_{(p-m)/2}&=\left(\int_{\real}|w_0(y)|^{(p-m)/2}\,dy\right)^{2/(p-m)}=\left(\int_{\real}|u_0(r)|^{(p-m)/2}\frac{1}{r}\,dr\right)^{2/(p-m)}\\
&=\left(\int_{\real}r^{-N}|u_0(r)|^{(p-m)/2}r^{N-1}\,dr\right)^{2/(p-m)}=\|u_0\|_{(p-m)/2;-N}<\delta.
\end{split}
\end{equation*}
According to \cite[Theorem 2]{Su98}, there exists a global in time solution $w$ to Eq. \eqref{interm19} with initial condition $w_0$ and a constant $K>0$ such that
$$
\|w(t)\|_{\infty}\leq Kt^{-1/(p-1)}, \qquad {\rm for \ any} \ t>0.
$$
We reach the conclusion by defining $u(r,t)=w(y,t)$, $r=e^y$.

\medskip

\noindent \textbf{Part 3.} It follows in a similar way as Part 2, by simply noticing (with a completely similar change of variable as above) that $\|u_0\|_{1;-N}=\|w_0\|_1$ and applying the first statement of \cite[Theorem 3]{Su98} to the initial condition $w_0(y)=u_0(r)$.
\end{proof}

\noindent \textbf{Remark.} All the results in \cite{Su98} are valid when the convection coefficient is nonzero, that is, in our case, in dimensions different from $N=2$. However, when $N=2$, the convection term disappears and we are left with the classical reaction-diffusion equation where much more is known (see for example Section \ref{sec.sigmaequal}). In particular, Theorem \ref{th.Eul.large} still holds true but it can be strongly improved, for example by filling the gap $m+1\leq p<m+2$ which enters the blow-up range, by standard Fujita-type results. We do not extend this discussion here, but one can obtain in the case $N=2$ self-similar solutions, blow-up rates and much more information on radially symmetric solutions to Eq. \eqref{eq1.gen}, as we already did in Section \ref{sec.sigmaequal}.

\bigskip

We are left with the case $m=1$, $\sigma=-2$, where, as explained at the end of Subsection \ref{subsec.other}, we can apply the transformation \eqref{change4}-\eqref{change4.bis} to reach a reversed Fisher-KPP equation \eqref{eq3.eul}, where now $\sigma_2\geq-2$ is completely independent. Of course, for $\sigma_2=-2$ all the above still holds true and much more, since we are left with the standard semilinear reaction-diffusion equation. We are thus interested in the case $\sigma_2\neq-2$. In this case, we can prove the following result.
\begin{theorem}\label{th.Eul.linear}
Let $m=1$, $p>1$ and $\sigma_1=-2$, $\sigma_2>-2$. We then have
\begin{enumerate}
  \item If $\sigma_2>(N-2)(p-1)-2$, then any nontrivial radially symmetric solution to Eq. \eqref{eq1.gen} blows up in finite time.
  \item If $N>2$ and $\sigma_2<(N-2)(p-1)-2$, then there exist radially symmetric solutions to Eq. \eqref{eq1.gen} which are either self-similar in exponential form \eqref{exponential} if $\sigma_2=-2$, or having an integrable singularity at $x=0$ if $\sigma_2>-2$.
\end{enumerate}
\end{theorem}
\begin{proof}
In the first case, notice that Eq. \eqref{eq3.eul} becomes an equation with double reaction, since
$$
N-2-\frac{\sigma_2+2}{p-1}=\frac{(N-2)(p-1)-\sigma_2-2}{p-1}<0.
$$
It is then easy to show that all nontrivial solutions to Eq. \eqref{eq3.eul} blow up in finite time, and the same occurs for the radially symmetric solutions to Eq. \eqref{eq1.gen}. In the second case, Eq. \eqref{eq3.eul} becomes a reversed Fisher-type equation and enters as a particular case of the more general study performed in the note \cite{SB05}, with $m=1$, $q=1$ and $p>1$ in the notation therein. Thus, \cite[Theorem 1]{SB05} entails that Eq. \eqref{eq3.eul} admits solutions in the form of traveling waves, namely
$$
\psi(\overline{y},t)=f(y+ct), \qquad c>0.
$$
By undoing first the transformation \eqref{change4.bis}, we obtain solutions to Eq. \eqref{interm3} having the form
$$
w(y,t)=\psi(y+Kt,t)=f(y+(K+c)t), \qquad K=N-2-\frac{2(\sigma_2+2)}{p-1}.
$$
We are now left with undoing the transformation \eqref{change4} for these solutions, to get
\begin{equation}\label{interm22}
u(x,t)=r^{-\delta}f(\ln\,r+(K+c)t)=|x|^{-(\sigma_2+2)/(p-1)}f(\ln(|x|e^{(K+c)t})).
\end{equation}
Notice that these solutions are self-similar in exponential form if $\sigma_2=-2$, whose profile is the composition between the traveling wave profile $f$ and the logarithmic function. If $\sigma_2>-2$, we obtain in \eqref{interm22} solutions that are singular at $x=0$, due to the fact that the traveling waves to Eq. \eqref{eq3.eul} are bounded as $y\to-\infty$, as established in \cite{SB05}. Since we are in the case when
$$
N>2+\frac{\sigma_2+2}{p-1}>\frac{\sigma_2+2}{p-1},
$$
we readily infer that the singular solutions given by \eqref{interm22} are integrable near $x=0$, as claimed.
\end{proof}

\bigskip

\noindent \textbf{Acknowledgements} R. I. and A. S. are partially supported by the Spanish project PID2020-115273GB-I00.

\bibliographystyle{plain}

\end{document}